\documentclass[11pt]{article}
\usepackage[utf8]{inputenc}
\usepackage[T1]{fontenc}
\usepackage[a4paper]{geometry}
\usepackage{microtype}
\usepackage{amsmath}
\usepackage{amssymb}
\usepackage{amsthm}
\usepackage{natbib}
\usepackage{hyperref}
\usepackage[capitalize]{cleveref}
\usepackage{csquotes}
\MakeOuterQuote{"}
\usepackage{fontawesome5}
\usepackage{comment}
\usepackage{mathtools}
\usepackage{dsfont}
\usepackage{enumitem}
\usepackage{calrsfs}
\usepackage{bm}
\usepackage{orcidlink}

\newtheorem{lemma}{Lemma}[section]
\newtheorem{assumption}{Assumption}
\newtheorem{theorem}[lemma]{Theorem}

\newtheorem{corollary}[lemma]{Corollary}
\newtheorem{proposition}[lemma]{Proposition}

\theoremstyle{definition}
\newtheorem{example}[lemma]{Example}
\newtheorem{remark}[lemma]{Remark}

\newcommand{\Ee}{\mathbb{E}}
\newcommand{\Mm}{\mathbb{M}}
\newcommand{\Xx}{\mathbb{X}}
\newcommand{\Yy}{\mathbb{Y}}
\newcommand{\Zz}{\mathcal{Z}}
\newcommand{\XY}{\Xx \times \Yy}

\newcommand{\Bc}{\mathcal{B}}
\newcommand{\Ec}{\mathcal{E}}

\newcommand{\Pc}{\mathcal{P}}

\newcommand{\expec}{\operatorname{\mathsf{E}}}
\newcommand{\pr}{\operatorname{\mathsf{P}}}
\newcommand{\Pemp}{\mathbb{P}_n}

\newcommand{\wto}{\rightsquigarrow}
\newcommand{\1}{\mathds{1}}

\newcommand{\bth}{\bar{\theta}}
\newcommand{\clB}{\overline{B}}

\newcommand{\ES}{\operatorname{ES}_\beta}
\newcommand{\LogS}{\operatorname{LogS}}

\newcommand{\reals}{\mathbb{R}}
\newcommand{\R}{\reals}
\newcommand{\Rd}{\reals^d}
\newcommand{\nat}{\mathbb{N}}
\newcommand{\diff}{\mathrm{d}}
\newcommand{\point}{\,\cdot\,}

\newcommand{\dlt}{\delta}
\newcommand{\eps}{\varepsilon}

\newcommand{\hetan}{\hat{\eta}_n}

\newcommand{\limsupn}{\limsup_{n\to\infty}}

\newcommand{\oh}{\operatorname{o}}

\newcommand{\closure}{\operatorname{cl}}

\renewcommand{\subset}{\subseteq}

\newcommand{\abr}[1]{\left|#1\right|}
\newcommand{\cbr}[1]{\left\{#1\right\}}
\newcommand{\nbr}[1]{\left\|#1\right\|}
\newcommand{\rbr}[1]{\left(#1\right)}
\newcommand{\sbr}[1]{\left[#1\right]}

\numberwithin{equation}{section}

\setlist{itemsep=0pt, topsep=5pt, partopsep = 5pt}

\title{
Consistency of M-estimators for non-identically distributed data: the case of fixed-design distributional regression
}
\author{
Axel B\"ucher\thanks{Ruhr-Universität Bochum, Fakultät für Mathematik. Email: \href{mailto:axel.buecher@rub.de}{axel.buecher@rub.de}} \orcidlink{0000-0002-1947-1617}
\and
Johan Segers\thanks{KU Leuven, Department of Mathematics, and UCLouvain, LIDAM/ISBA. Email: \href{mailto:jjjsegers@kuleuven.be}{jjjsegers@kuleuven.be}}
\orcidlink{0000-0002-0444-689X}
\and
Torben Staud\thanks{Ruhr-Universität Bochum, Fakultät für Mathematik. Email: \href{mailto:torben.staud@rub.de}{torben.staud@rub.de}}
\orcidlink{0009-0004-4526-8585}
}
\date{\today}
\begin{document}
\maketitle

\begin{abstract}
This paper explores strong and weak consistency of M-estimators for non-identically distributed data, extending prior work. 
Emphasis is given to scenarios where data is viewed as a triangular array, which encompasses distributional regression models with non-random covariates. 
Primitive conditions are established for specific applications, such as estimation based on minimizing empirical proper scoring rules or conditional maximum likelihood.
A key motivation is addressing challenges in extreme value statistics, where parameter-dependent supports can cause criterion functions to attain the value $-\infty$, hindering the application of existing theorems.

\end{abstract}

\section{Introduction} 
\label{sec:introduction}

M-estimators are a broad class of estimators that arise from maximizing an objective function that corresponds to a sample average. More specifically, given observations $z_1, \dots, z_n$, they arise from (approximately) maximizing a function of the form
\begin{align}
\label{eq:criterion-introduction}
\eta \mapsto M_n(\eta) = \frac1n \sum_{i=1}^n m_\eta(z_i)
\end{align}
over a given parameter set $H$, where $m_\eta$ is a known function taking values in $[-\infty, \infty)$. The framework is broad enough to cover as special cases maximum likelihood estimators from classical parametric statistics or empirical risk minimizers from supervised learning problems.
Asymptotic theory, including weak or strong consistency as well as asymptotic normality, has been studied extensively, in particular for the case of independent and identically distributed observations. For an overview, we refer to Section 5 in \cite{van1998asymptotic} and the references therein.

The present paper is concerned with strong (and weak) consistency of M-estimators in situations where the data is not identically distributed, thereby providing extensions of Sections 5.1 and 5.2 in \cite{van1998asymptotic}. Such situations commonly occur in (distributional) regression models in the fixed-design setting \citep{KneibSilbersdorffSaefken2023, Klein2024}, that is, with non-random covariates, a classical example being the ordinary least squares estimator in linear models. Respective asymptotic theory is usually based on viewing the data as a triangular array, and we adopt this perspective as well; see, for instance, \cite{Liese1994}, \cite{berlinet2000necessary}, and \cite{SalibianBarrera2006} for models where the parameter only enters through a location term. A general framework for deriving strong consistency is offered by \citet{lachout2005strong}, but their assumptions are high-level. It is one the main purposes of this paper to provide more primitive conditions and model assumptions. 

A specific need has motivated this paper: common parametric distributional regression models in extreme value statistics have parameter-dependent supports; see \cite{Phi20} for applications in extreme-event attribution and \cref{subsec:conditional-mle} below for mathematical details. Consequently, the criterion function of any likelihood-based estimator may attain the value $-\infty$, hindering, for example, the application of Theorem~2.2 in \citet{lachout2005strong}. 
Despite acknowledgment of this issue (see, e.g., the discussion of Assumption C.1 in \citealp{Padoan10}), a formal consistency proof for conditional maximum likelihood estimators in extreme value statistics remains absent to the best of our knowledge.

General consistency results for M-estimators in triangular arrays are presented in \cref{sec:consistency-general} and adapted to distributional regression models with non-random covariates in \cref{sec:consistency-covariates}. Specific applications are detailed in \cref{sec:applications}, including estimators based on optimizing empirical proper scoring rules \citep{gneiting2007strictly, dawid2016minimum} and (conditional) maximum likelihood estimators. Additionally, we examine conditional maximum likelihood estimators for heavy-tailed regression models for block maxima. The corresponding consistency result, outlined in \cref{subsec:block-maxima}, can be viewed as an extension of the main finding in \cite{Dom15} to a conditional framework, focusing on the heavy-tailed case, as in \cite{BucSeg18}. Furthermore, \cref{sec:argmax} provides adaptations of Theorem 2.2 from \citet{lachout2005strong} suitable for settings where the criterion function may attain the value $-\infty$. These adaptations could be collectively termed `argmax theorems without uniform convergence', with the main proof idea going back to \cite{wald1949note}. Finally, all proofs are provided in \cref{sec:proofs}.

Throughout, the arrow $\wto$ denotes weak convergence of probability measures on a metric space. For positive integer $n$, we write $[n]=\{1, \dots, n\}$.

\section{Consistency of M-estimators in triangular arrays}
\label{sec:consistency-general}

As motivated in the introduction,  we consider a data-generating process corresponding to a sufficiently regular triangular array of random variables.

\begin{assumption}[Data-generating process]
    \label{ass:dgpZ}
    The metric space $(\Zz, d_{\Zz})$ is separable, and $(Z_{n,i}:n\in\nat,i\in[n])$ is a rowwise independent triangular array of random variables in $\Zz$, all defined on the same probability space. The distributions $P_{n,i}$ of $Z_{n,i}$ satisfy $n^{-1}\sum_{i=1}^n P_{n,i} \wto P_Z$ as $n\to\infty$ for some distribution $P_Z$ on $\Zz$.   
\end{assumption}

The condition that the  random variables $Z_{n,i}$ are all defined on the same probability space is only needed for strong consistency (\cref{thm:strongZ}). For weak consistency (\cref{thm:weakZ}), it is sufficient that within each row $n$, the variables $(Z_{n,i}: i \in [n])$ are defined on the same probability space.
For simplicity, we will disregard this distinction henceforth.

The subsequent condition pertains to the known criterion function $m(\eta, z) = m_\eta(z)$ informally introduced in \eqref{eq:criterion-introduction}.

\begin{assumption}[Criterion function]
    \label{ass:critZ}
    The metric space $(H, d_H)$ is compact and the function $m : H \times \Zz \to [-\infty,\infty)$ is upper semicontinuous. For all $\eta \in H$, the function $m_\eta$ is $P_Z$-quasi-integrable. 
\end{assumption}

While compactness of $H$ is vital for our proofs, the results remain applicable even when $H$ is not compact, provided suitable adjustments are made. Specifically, one can either additionally show that the estimator eventually maps into a compact set, or replace $H$ with an appropriate compactification. The latter may present additional challenges, potentially resolvable through an observation-grouping technique as outlined in Exercise 5.25 of \cite{van1998asymptotic}.

For $\eta \in H$ fixed, we will write $m_\eta : \Zz \to [-\infty,\infty)$ for the function $z \mapsto m(\eta,z)$. The random and asymptotic criterion functions $M_n,M : H \to [-\infty,\infty)$ are defined as
\begin{align}
    \label{eq:MnZ}
    M_n(\eta) &= \frac{1}{n} \sum_{i=1}^n m_\eta(Z_{n,i}), \\
    \label{eq:MZ}
    M(\eta) &= \int_{\Zz} m_\eta(z) \, P_Z(\diff z) = P_{Z} m_\eta.
\end{align}
Informally, $M_n(\eta)$ can be considered as an estimator for $M(\eta)$, such that any (approximate) maximizer of $M_n$ can be regarded as an (approximate) maximizer of $M$. The following assumption concerns the latter maximization problem.

\begin{assumption}[True parameter]
    \label{ass:trueZ}
    The function $M$ has a unique point of finite maximum $\eta_0 \in H$:
    \[
        \forall \eta \in H \setminus \{\eta_0\}:
        \qquad
        -\infty \le M(\eta) < M(\eta_0) < \infty.
    \]
    The function $m_{\eta_0}$ is continuous $P_Z$-almost everywhere.
\end{assumption}

The aforementioned proximity of $M_n(\eta)$ to $M(\eta)$ necessitates a suitable (weak or strong) law of large numbers. Notably, for triangular arrays, strong consistency demands non-standard results which require more than mere (uniform) integrability \citep[p.~215]{chandra1992cesaro}. A sufficient condition is $L^2$-stochastic dominance, as detailed in \cite{HuMorTay89}:
a family of random variables $(X_\alpha)_{\alpha \in A}$ on the real line is said to be \emph{$L^p$-stochastically dominated}, with $1 \le p < \infty$, if there exists a nonnegative random variable $W$ with $\expec[W^p] \le \infty$ such that 
\begin{equation}
    \label{eq:def:stochdom}
    \forall t \in [0, \infty): \qquad 
    \sup_{\alpha \in A} \pr(|X_\alpha|>t) 
    \le \pr(W > t). 
\end{equation}
The family is called \emph{uniformly integrable} if
\begin{equation}
    \label{eq:def:uniformly-integrable}
    \lim_{c\to\infty} \sup_{\alpha \in A} \expec[ |X_\alpha|  \1_{\{|X_\alpha|>c\}}] = 0.
\end{equation}
By Markov's inequality, a sufficient condition for $L^p$-stochastic dominance is the existence of $\delta > 0$ such that 
$
\sup_{\alpha \in A} \expec[|X_\alpha|^{p+\delta}]<\infty,
$
and a sufficient condition for uniform integrability is $L^1$-stochastic dominance.
For $B \subset H$, put $m_B(z) = \sup_{\eta \in B} m_\eta(z)$. Let $\clB(\eta,\delta)$ be the closed ball in $H$ with center $\eta \in H$ and radius $\delta>0$. Let $(x)^+ = \max(x,0)$ denote the positive part of $x \in \reals$.

\begin{assumption}[$L^2$-stochastic dominance]
    \label{ass:stochdomZ}
    \begin{itemize}
    \item[(i)] There exists $n_0 \in \nat$ such that the family $\cbr{m_{\eta_0}(Z_{n,i}) : n \ge n_0, i \in [n]}$ is $L^2$-stochastically dominated. 
    \item[(ii)] For every $\eta \in H \setminus \cbr{\eta_0}$ there exists $\delta >0$ and $n_0 \in \nat$ such that $P_Z m_{\clB(\eta,\delta)} < \infty$ and the family $\{(m_{\clB(\eta,\delta)}(Z_{n,i}))^+: n \ge n_0, i \in [n]\}$ is $L^2$-stochastically dominated.
    \end{itemize}
\end{assumption}

Often, the function $m$ is upper bounded on $H \times \Zz$, that is,
\begin{align} \label{eq:bounded-m}
\sup_{\eta \in H, z \in \Zz} m_\eta(z) < \infty.
\end{align}
Under this condition, the $L^2$-stochastic dominance condition in (ii) is immediate, while the condition in (i) becomes a condition on the lower tail of $m_{\eta_0}(Z_{n,i})$ only. 

\begin{theorem}[Strong consistency]
    \label{thm:strongZ}
    If Assumptions~\ref{ass:dgpZ}, \ref{ass:critZ}, \ref{ass:trueZ} and \ref{ass:stochdomZ} hold, then every estimator sequence $\hetan$ defined on the same probability space as the array $(Z_{n,i})_{n,i}$ that satisfies $M_n(\hetan) \ge M_n(\eta_0) - \oh(1)$ almost surely as $n\to\infty$ is strongly consistent for $\eta_0$.
\end{theorem}

Weak consistency of $\hat \eta_n$ can be derived under a weaker integrability condition than $L^2$-stochastic dominance. Again, if $m$ satisfies \cref{eq:bounded-m}, item (ii) in the next condition is immediate and item (i) only concerns the lower tail of $m_{\eta_0}(Z_{n,i})$.

\begin{assumption}[Uniform integrability]
    \label{ass:uiZ}
    \begin{itemize}
    \item[(i)] There exists $n_0 \in \nat$ such that the family $\cbr{m_{\eta_0}(Z_{n,i}) : n \ge n_0, i \in [n]}$ is uniformly integrable. 
    \item[(ii)] For every $\eta \in H \setminus \cbr{\eta_0}$ there exists $\delta >0$ and $n_0 \in \nat$ such that $P_Z m_{\clB(\eta,\delta)} < \infty$ and such that the family $\{(m_{\clB(\eta,\delta)}(Z_{n,i}))^+ : n \ge n_0, i \in [n]\}$ is uniformly integrable.
    \end{itemize}
\end{assumption}

\begin{theorem}[Weak consistency]
    \label{thm:weakZ}
    If Assumptions~\ref{ass:dgpZ}, \ref{ass:critZ}, \ref{ass:trueZ} and \ref{ass:uiZ} hold, then every estimator sequence $\hetan$ such that $\hetan$ is defined on the same probability space as the tuple $(Z_{n,i})_{i\in[n]}$ and that satisfies $M_n(\hetan) \ge M_n(\eta_0) - \oh_{\pr}(1)$ as $n\to\infty$ is weakly consistent for $\eta_0$.
\end{theorem}

\begin{remark}
By the mere definition of weak consistency, the conclusion in \cref{thm:weakZ} does not depend on the dependence between rows in the triangular array. This is different in \cref{thm:strongZ}, and it is perhaps surprising that the conclusion in \cref{thm:strongZ} is true \textit{for any} form of stochastic dependence/independence between rows. The price to be paid for this generality is a stronger integrability condition. 
\end{remark}

\section{Consistency of M-estimators in conditional models with non-random covariates}
\label{sec:consistency-covariates}

Throughout this section, we consider M-estimators in conditional models that are based on $n$ observations from a tuple $(x,Y)$, with the distribution of the response $Y$ depending on a non-random covariate $x$.  In many related situations such as nonparametric regression or classical linear regression, deriving asymptotic results is most naturally framed by viewing the data as a triangular array. We adopt this perspective and impose conditions that enable us to apply the general results from the previous section.

Specifically, we assume that $(x,Y)$ takes values in $\XY$, where $(\Xx, d_\Xx)$ and $(\Yy, d_\Yy)$ are separable metric spaces equipped with the respective Borel $\sigma$-fields $\Bc(\Xx)$ and $\Bc(\Yy)$, respectively. The distribution of the response variable is assumed to depend on $x$, which will formally be described by the notion of a Markov kernel, that is, by a map $Q : \Xx \times \Bc(\Yy) \to [0, 1]$ such that $Q(x, \point)$ is a probability measure on $\Yy$ for each $x \in \Xx$ and such that $x \mapsto Q(x, B)$ is measurable for each $B \in \Bc(\Yy)$.
Despite the fact that $x$ is deterministic, it helps to think of $Q(x, \point)$ as the conditional distribution of the response variable taking values in $\Yy$ given the covariate $x$; an alternative notation to keep in mind would be $P_{Y \mid X=x}(\point)$ instead of $Q(x, \point)$. 

In a nutshell, we will apply the framework of \cref{sec:consistency-general} to $\Zz = \Xx \times \Yy$ and $Z_{n,i} = (x_{n,i}, Y_{n,i})$ in \cref{ass:dgpZ-conditional} below. The same approach could be followed to handle random covariates, putting $Z_{n,i} = (X_{n,i}, Y_{n,i})$; we do not develop this further in this paper.

\begin{assumption}[Data-generating process]
    \label{ass:dgpZ-conditional}
    For sample size $n \in \nat$, we observe a finite sequence $(x_{n,i}, Y_{n,i})_{i\in[n]}$, where $x_{n,1}, \dots, x_{n,n}$ are deterministic elements in $\Xx$ and where $Y_{n,1},\dots, Y_{n,n}$ are independent $\Yy$-valued random variables such that $Y_{n,i} \sim Q_n(x_{n,i}, \diff y)$ for all $i \in [n]$ and for some Markov kernel $Q_n: \Xx \times \Bc(\Yy) \to [0, 1]$. Moreover:
    \begin{enumerate}[label = (\alph*)]
    \item The observations $x_{n,1}, \dots, x_{n,n}$ are regular in the sense that there exists a non-degenerate probability distribution $P_X$ on $\Xx$ such that $n^{-1}\sum_{i=1}^n \delta_{x_{n,i}} \wto P_X$ as $n \to \infty$.
    \item The distribution of $Y_{n,1},\dots, Y_{n,n}$ is regular in the sense that there exists a Markov kernel $Q: \Xx \times \Bc(\Yy) \to [0, 1]$ such that $Q_n(\point, \diff y )$ converges continuously to $Q(\point, \diff y)$ in the weak topology, that is, for any sequence $x_n \to x$ in $\Xx$, the measures $Q_n(x_n, \diff y )$ converge weakly to $Q(x, \diff y)$. Moreover, for the limiting kernel $Q$, the function $x \mapsto Q(x, \diff y)$ is weakly continuous, i.e., if $x_n \to x$ in $\Xx$, then $Q(x_n, \diff y) \wto Q(x, \diff y)$ as $n \to \infty$.
    \end{enumerate}
\end{assumption}

A typical example for part~(a) of \cref{ass:dgpZ-conditional}, often studied in non-parametric regression, is the simple uniform design $x_{i,n}=i/n$, which could in practice for instance correspond to rescaled time. It has also been imposed in \cite{Liese1994, berlinet2000necessary} for studying M-estimators in general mean regression models. Moreover, it is comparable to standard assumptions for consistency statements on the ordinary least squares estimator in linear mean regression models with non-random design matrix $\bm X_n=(x_{n,1}^\top, \dots, x_{n,n}^\top)^\top$, which is typically assumed to satisfy $\lim_{n \to \infty} \frac1n \sum_{i=1}^n x_{n,i}x_{n,i}^\top = \bm S$ for some matrix $\bm S \in \R^{k \times k}$. Under \cref{ass:dgpZ-conditional}(a), we would have $\bm S= \int x x^\top P_X(\diff x)$, provided $P_X$ has finite cross-moments of some order $p>2$. 

An important special case for part~(b) of \cref{ass:dgpZ-conditional} arises when $Q_n=Q$ does not depend on $n$. In this case, the condition simplifies to requiring only weak continuity of the mapping $x \mapsto Q(x, \diff y)$; this covers most parametric distributional regression models like $Q(x, \point) = \mathcal N(f_\theta(x), \sigma^2)$ for some smooth, parametric regression function $f_\theta$ and some $\sigma^2>0$. More specific examples will be worked out in Section~\ref{sec:applications}.

We start by showing that \cref{ass:dgpZ-conditional} implies \cref{ass:dgpZ}. Consistency of M-estimators is then an immediate corollary.

\begin{lemma}\label{lem:from-conditional-to-triangular}
    Suppose that \cref{ass:dgpZ-conditional} is met. Then \cref{ass:dgpZ} is met with $\Zz=\XY$, $Z_{n,i}=(x_{n,i}, Y_{n,i})$,
    $
        P_{n,i}(\diff (x,y) ) 
        = 
        Q_n(x, \diff y) \, \delta_{x_{n,i}}(\diff x)
    $
    and $P_Z=P_{X,Y}$, where
    \begin{equation}
        \label{eq_ext:Pxy}
        P_{X, Y}(\diff(x,y)) 
        = 
        Q(x, \diff y) P_X(\diff x)
    \end{equation}
    and where $\delta_x$ is the Dirac-measure at $x$.
\end{lemma}

\begin{corollary}
    \label{cor:conditional:strong}
    Suppose that \cref{ass:dgpZ-conditional} holds. If, additionally, Assumptions~\ref{ass:critZ}, \ref{ass:trueZ} and \ref{ass:stochdomZ}  are met for $\Zz=\XY$, $Z_{n,i}=(x_{n,i}, Y_{n,i})$ and $P_Z=P_{X,Y}$, then every estimator sequence $\hetan$ defined on the same probability space as the array $(x_{n,i}, Y_{n,i})_{n,i}$ that satisfies $M_n(\hetan) \ge M_n(\eta_0) - \oh(1)$ almost surely as $n\to\infty$ is strongly consistent for $\eta_0$.
\end{corollary}

\begin{corollary}
    \label{cor:conditional:weak}
    Suppose that \cref{ass:dgpZ-conditional} holds. If, additionally, Assumptions~\ref{ass:critZ}, \ref{ass:trueZ} and \ref{ass:uiZ} are met for $\Zz=\XY$, $Z_{n,i}=(x_{n,i}, Y_{n,i})$ and $P_Z=P_{X,Y}$, then every estimator sequence $\hetan$ such that $\hetan$ is defined on the same probability space as $(x_{n,i}, Y_{n,i})_{i\in[n]}$ and that satisfies $M_n(\hetan) \ge M_n(\eta_0) - \oh_{\pr}(1)$ as $n\to\infty$ is weakly consistent for $\eta_0$.
\end{corollary}

\section{Applications to optimum score and conditional maximum likelihood estimators}
\label{sec:applications}

As outlined in the introduction, the results from the previous sections can be applied to various common estimators. We start by dealing with optimum scoring rule estimation, then adapt these results to (conditional) maximum likelihood estimation, and finally examine a pseudo-maximum likelihood estimator in a specific heavy-tailed regression model for block maxima. While the dependence of the distribution of $Y_{n,i}$ on $n$ only enters through $x_{n,i}$ in the first two applications (i.e., $Q_n=Q$ in \cref{ass:dgpZ-conditional}), a more general dependence is needed for the latter application.

\subsection{Optimum score estimators with non-random covariates}
\label{subsec:optimum-score}

We specialize the set-up of Section~\ref{sec:consistency-covariates} to the case of optimum score estimators. Recall that $\Xx$ and $\Yy$ are separable metric spaces, while the parameter space $H$ is compact.

\begin{assumption}[Parametric distributional regression model and sampling scheme]
    \label{ass:paramodel}
    Let $\{x_{n,i} \colon n \in \nat, i \in [n] \} \subset \Xx$ be a deterministic triangular array satisfying \cref{ass:dgpZ-conditional}(a) and let $\{ Y_{n,i}\}_{i,n}$ be a rowwise independent $\Yy$-valued triangular array such that
    $
        Y_{n,i} \sim P_{\bth(x_{n,i}, \eta_0)},
    $
    where $(P_\theta \colon \theta \in \Theta)$ is a known parametric family of distributions, $\bth \colon \Xx \times H \to \Theta$ is a known link function and $\eta_0 \in H$ is the true unknown parameter of interest. The function $\bth$ is continuous while the map $\theta \mapsto P_\theta$ is one-to-one (i.e., the parameter $\theta$ is identifiable) and weakly continuous. Finally, $\eta_0$ is identifiable in the sense that, for the limit distribution $P_X$ in \cref{ass:dgpZ-conditional}(a), we have
    \begin{equation}
        \label{eq:eta-identifiable}
        \forall \eta \in H \setminus \cbr{\eta_0} : \qquad 
        P_X\rbr{\cbr{x \in \Xx : \bth(x,\eta) \ne \bth(x,\eta_0)}} > 0. 
    \end{equation}
\end{assumption}

In the notation of Section~\ref{sec:consistency-covariates}, the Markov kernel $Q_n(x,\point)=Q(x,\point)=P_{\bth(x, \eta_0)}(\point)$ does not depend on $n$. Moreover, the distribution $P_{X,Y}$ from \cref{eq_ext:Pxy} is 
\begin{equation}
    \label{eq:PXYeta0}
    P_{X,Y}(\diff (x, y)) = P_{\bth(x,\eta_0)}(\diff y) \, P_X(\diff x). 
\end{equation}

Let $\Pc$ be a convex set of probability measures on $\Yy$. Given an outcome $y \in \Yy$  we wish to attach a score to a probabilistic forecast $P \in \Pc$, with higher scores being better. A \emph{strictly proper scoring rule} is a map $S : \Pc \times \Yy \to [-\infty,\infty]$ with the following properties: 
\begin{itemize}
\item[(i)] For every $P \in \Pc$, the map $y \mapsto S(P, y)$ is Borel measurable.
\item[(ii)] For every $P, P' \in \Pc$, the integral $S(P', P) = \int_{\Yy} S(P', y) \, P(\diff y)$ exists in $[-\infty,\infty]$.
\item[(iii)] For every $P, P' \in \Pc'$, we have $S(P, P) \ge S(P', P)$, with strict inequality if $P \ne P'$.
\end{itemize}
We refer to \cite{gneiting2007strictly, dawid2016minimum} for further context and references. 

\begin{assumption}[Scoring rule]
    \label{ass:scoring}
    $S : \Pc \times \Yy \to [-\infty,\infty]$ 
    is a strictly proper scoring rule on $\Pc$, a convex set of probability measures on $\Yy$ containing $P_\theta$ for all $\theta \in \Theta$, and satisfies $S(P_\theta, y)<\infty$ for all $(\theta,y) \in \Theta \times \Yy$. The map $(\theta,y) \mapsto S(P_\theta,y)$ is upper semicontinuous and the map $(x,y) \mapsto S(P_{\bth(x,\eta_0)},y)$ is continuous $P_{X,Y}$-almost everywhere. 
\end{assumption}

For $\eta \in H$, define $m_\eta : \XY \to [-\infty, \infty)$ by
$
    m_\eta(x,y) = S(P_{\bth(x,\eta)}, y).
$
The random and limiting criterion functions in \cref{eq:MnZ,eq:MZ} become
\begin{align}
    \label{eq:score:Mn}
    M_n(\eta) 
    &= \frac{1}{n} \sum_{i=1}^n S(P_{\bth(x_{n,i},\eta)}, Y_{n,i}), \\
    \nonumber
    M(\eta) 
    &= \int_{\Xx} \int_{\Yy} 
        S(P_{\bth(x,\eta)}, y) \, 
    P_{\bth(x,\eta_0)}(\diff y) \, P_X(\diff x) 
    = \int_{\Xx} S(P_{\bth(x,\eta)}, P_{\bth(x,\eta_0)}) \, P_X(\diff x).
\end{align}
We will prove below that the function $\eta \mapsto M(\eta)$ is well defined (i.e., for every $\eta\in H$, the map $x \mapsto S(P_{\bth(x, \eta)}, P_{\bth(x, \eta_0)})$ is quasi-integrable with respect to $P_X$) as a consequence of the following stochastic dominance condition; see \cref{eq:def:stochdom}. 

\begin{assumption}[$L^2$-stochastic dominance for scoring rules]
\label{ass:scoring-dominance}
\begin{itemize}
    \item[(i)] The family of random variables $\{S(P_\theta,Y): \theta \in \Theta, Y \sim P_\theta \}$ is $L^2$-stochastically dominated.
    \item[(ii)] For every $\eta \in H \setminus \cbr{\eta_0}$ there exists $\delta > 0$ such that 
    the family of random variables
    \begin{equation}
    \label{eq:Setadeltafamily}
        \Big\{
            \sup_{\bar{\eta}\in\clB(\eta,\delta)} 
                \bigl(S(P_{\bth(x,\bar{\eta})}, Y)\bigr)^+ : 
                x \in \Xx, Y \sim P_{\bth(x,\eta_0)}
        \Big\} 
    \end{equation}
    is $L^2$-stochastically dominated. 
\end{itemize}
\end{assumption}

\begin{theorem}
    \label{thm:scoring:strong}
    If Assumptions~\ref{ass:paramodel}, \ref{ass:scoring}, and \ref{ass:scoring-dominance} are met, then every estimator $\hat \eta_n$ satisfying $M_n(\hat \eta_n) \geq M_n(\eta_0) - \oh_{a.s.}(1)$ is strongly consistent for $\eta_0$.
\end{theorem}

For the corresponding result about weak consistency, it is sufficient to replace the words ``$L^2$-stochastically dominated'' in \cref{ass:scoring-dominance} by ``uniformly integrable'' and replace the $\oh_{a.s.}(1)$ term in \cref{thm:scoring:strong} by $\oh_{\pr}(1)$.

\begin{example}[Optimum energy score estimation] \label{ex:energy-score}
For positive integer $d$ and for $0 < \beta < 2$, let $\Pc_\beta(\Rd)$ be the set of probability measures $P$ on $\Rd$ such that $\int_{\Rd} \nbr{y}^\beta \, P(\diff y) < \infty$, where $\nbr{\point}$ denotes the Euclidean norm. The \emph{energy score} \citep{szekely2003E} $\ES : \Pc_\beta(\Rd) \times \Rd \to \reals$ is defined as
\[
    \ES(P, y) 
    = \frac{1}{2} \expec \bigl[ \nbr{Y-Y'}^\beta \bigr] - \expec\bigl[\nbr{Y-y}^\beta\bigr],
\]
where the random vectors $Y$ and $Y'$ are independent and have distribution $P$. The energy score is strictly proper relative to $\Pc_\beta(\Rd)$ thanks to \citet[Theorem~1]{szekely2003E}; see also \citet[Sections~4.3 and~5.1]{gneiting2007strictly}. The special case $d = 1$ and $\beta = 1$ yields the Continuous Ranked Probability Score (CRPS). The energy score is itself a special case of the \emph{kernel score} introduced in \citet{dawid2006geometry-published}, see also \citet[Section~5.1]{gneiting2007strictly}.

Suppose that \cref{ass:paramodel} is met with $\Yy = \R^d$ and with $P_\theta \in \Pc_\beta(\Rd)$ for all $\theta \in \Theta$. Further assume that there exists $\delta > 0$ such that
    \begin{align} \label{eq:energy-moments}
        \sup_{\theta \in \Theta} \int_{\Rd} 
            \nbr{y}^{2\beta+\delta} \, 
        P_\theta(\diff y) < \infty.
    \end{align}
Then every optimum energy score estimator $\hat \eta_n$ satisfying $M_n(\hat \eta_n) \geq M_n(\eta_0) - \oh_{a.s.}(1)$ is strongly consistent for $\eta_0$, where
\begin{equation*}
    M_n(\eta) 
    = \frac{1}{n} \sum_{i=1}^n \ES(P_{\bth(x_{n,i},\eta)}, Y_{n,i}),
    \qquad \eta \in H,
\end{equation*}
See \cref{sec:applications-proofs} for the proof.
A weak consistency statement only requires \cref{eq:energy-moments} with the exponent $2\beta+\delta$ replaced by $\beta+\delta$; we omit further details.
\end{example}

\subsection{Conditional maximum likelihood estimation}
\label{subsec:conditional-mle}

Conditional maximum likelihood estimators arise as a special case of optimum score estimators in \cref{subsec:optimum-score}. Specifically, we will work under \cref{ass:paramodel} and additionally assume that the parametric model $(P_\theta : \theta \in \Theta)$ is dominated by a $\sigma$-finite measure $\nu$ on $\Yy$ with densities $p_\theta = \diff P_\theta/\diff \nu$ for $\theta \in \Theta$.
As scoring rule in \cref{ass:scoring} we consider the \emph{logarithmic score}
\begin{equation}
    \label{eq:LogS}
    \LogS(P, y) = \log\Big(\frac{\diff P}{\diff \nu}(y)\Big)  
    \qquad P \in \Pc, \, y \in \Yy,
\end{equation}
which is strictly proper with respect to the family $\Pc$ of probability measures on $\Yy$ dominated by $\nu$; see \citet[Section~4.1]{gneiting2007strictly} for context and references. If $P=P_\theta$ for some $\theta \in \Theta$, we also write
\[
\LogS(P, y) = \log p_\theta(y) = \ell_\theta(y),
\]
such that $m_\eta(x,y) = \ell_{\bth(x,\eta)}(y)$ and such that the criterion function $M_n$ in \cref{eq:score:Mn} is equal to the conditional log-likelihood function 
\[
    M_n(\eta) 
    = \frac{1}{n} \sum_{i=1}^n \ell_{\theta(x_{n,i}, \eta)}(Y_{n,i}).
\]
Note that $y$ does not need to belong to the support of $P_\theta$, so that the value $\ell_\theta(y) = - \infty$ is allowed; since $\ell_\theta(y) = +\infty$ can only occur on a $\nu$-null set, we can and will exclude it altogether. A maximiser $\hetan$ of $\eta \mapsto M_n(\eta)$ is a \emph{conditional maximum likelihood estimator}. Under the following set of conditions, its consistency is an immediate corollary of \cref{thm:scoring:strong}.

\begin{assumption}
    \label{ass:conditional-mle}
    \cref{ass:paramodel} holds and the parametric model $(P_\theta : \theta \in \Theta)$ is dominated by a $\sigma$-finite measure $\nu$ on $\Yy$ with densities $p_\theta = \diff P_\theta/\diff \nu$ for $\theta \in \Theta$. Moreover: 
    \begin{enumerate}[label=(\roman*)]
        \item the map $(\theta,y) \mapsto p_\theta(y)$  takes values in $[0,\infty)$ and is upper semicontinuous;
        \item the map $(x,y) \mapsto p_{\bth(x, \eta_0)}(y)$ is continuous $P_{X,Y}$-almost everywhere;
        \item the family $\{ \ell_\theta(Y) : \theta \in \Theta, Y \sim P_\theta\}$ is $L^2$-stochastically dominated;
        \item for every $\eta \in H \setminus \cbr{\eta_0}$ there exists $\delta > 0$ such that $\{\sup_{\bar{\eta}\in\clB(\eta,\delta)} (\ell_{\bth(x,\bar{\eta})}(Y))^+ : x \in \Xx, Y \sim P_{\bth(x,\eta_0)}\}$ is $L^2$-stochastically dominated. 
    \end{enumerate}
\end{assumption}

The most complicated condition is (iv). A simple and useful sufficient condition is $\sup_{\theta \in \Theta, y \in \Yy} p_\theta(y) < \infty$; see, for instance, \cref{ex:cmle_gev,ex:cmle_gpd} below.

\begin{corollary}
    \label{cor:cmle:strong}
    If \cref{ass:conditional-mle} holds, then every estimator sequence $\hetan$ such that $M_n(\hetan) \ge M_n(\eta_0) - \oh_{a.s.}(1)$ as $n\to\infty$ is strongly consistent for $\eta_0$.
\end{corollary}

As before, weak consistency already holds if $L^2$-stochastic dominance in \cref{ass:conditional-mle}(iii) and (iv) is replaced by uniform integrability and if the $\oh_{a.s.}(1)$ term in \cref{cor:cmle:strong} is replaced by an $\oh_{\pr}(1)$ term.

\cref{cor:cmle:strong} establishes consistency for estimators in models with covariate-dependent parameters, which are central in extreme value analysis and its applications. A key challenge in such models is that the support of the distribution depends on the parameter itself, rendering classical maximum likelihood estimation theory inapplicable.

\begin{example}[Generalized extreme value distribution with covariate-dependent parameters]
\label{ex:cmle_gev}
    The density function of the generalized extreme value (GEV) distribution on $\Yy = \reals$ with location-scale-shape parameter vector $\theta = (\mu,\sigma,\xi) \in \reals \times (0,\infty) \times \reals$ is
    \[
        p_{\theta}(y) = 
        \begin{cases}
        \frac{1}{\sigma} \rbr{1 + \xi \frac{y - \mu}{\sigma}}^{-1-1/\xi} 
            \exp \cbr{ -\rbr{1+ \xi \frac{y-\mu}{\sigma}}^{-1/\xi} }, 
            & \text{if $\xi \ne 0$, $1 + \xi \frac{y-\mu}{\sigma} > 0$}, \\[1ex]
        \frac{1}{\sigma} e^{- (y-\mu)/\sigma} \exp\rbr{-e^{- (y-\mu)/\sigma}},
            & \text{if $\xi = 0$,} \\[1ex]
        0, & \text{otherwise.}
        \end{cases}
    \]
    The support of the GEV distribution is 
    \[ 
        \cbr{y \in \reals : \sigma + \xi(y-\mu) > 0} =
        \begin{cases}
            (\mu - \sigma/\xi, \infty) & \text{if $\xi > 0$,} \\
            \reals & \text{if $\xi = 0$,} \\
            (-\infty, \mu - \sigma/\xi) & \text{if $\xi < 0$,}
        \end{cases}
    \]
    and depends on the parameter vector $\theta$. 
    Let $\Xx$ and $H$ be compact metric spaces and let the link function $\bth : \Xx \times H \to \reals \times (0, \infty) \times \reals$ be continuous and such that $\Theta := \bth(\Xx \times H) \subset \reals \times (0, \infty) \times (-1,\infty)$, that is, the shape parameter is restricted to $\xi > -1$.
    In contrast to the general theory in the preceding sections, we assume that the covariate space $\Xx$ is compact rather than just separable, and this in order to deal with the $L^2$-stochastic dominance condition in \cref{ass:conditional-mle}, which involves suprema over $\theta \in \Theta$ and $x \in \Xx$.
    Under the set-up of \cref{ass:paramodel} with true $\eta_0 \in H$, the conditions in \cref{ass:conditional-mle} hold. As a consequence, the conditional maximum likelihood estimator in $H$ is strongly consistent for the regression parameter $\eta_0$. Note that distributional regression models for the GEV-distribution are very common in extreme-event attribution; see \cite{Phi20} and the references therein. 
\end{example}

\begin{example}[Generalized Pareto distribution with covariate-dependent parameters]
\label{ex:cmle_gpd}
    The generalized Pareto (GP) density with scale-shape parameter vector $\theta = (a, \xi) \in (0, \infty) \times \reals$ on $\Yy = (0, \infty)$ has density function
    \[
        p_\theta(y) =
        \begin{cases}
            a^{-1} \rbr{1 + \xi y / a}^{-1-1/\xi}
            & \text{if $\xi \ne 0$, $1 + \xi y / a > 0$,} \\
            a^{-1} e^{-y/a}
            & \text{if $\xi = 0$,} \\
            0 & \text{otherwise.}
        \end{cases}
    \]
    As for the GEV distribution, the support of the GP distribution depends on the parameter:
    \[
        \cbr{y \in (0, \infty) : a + \xi y > 0} =
        \begin{cases}
            (0, \infty) & \text{if $\xi \ge 0$,} \\
            (0, -a/\xi) & \text{if $\xi < 0$.}
        \end{cases}
    \]
    As in \cref{ex:cmle_gev}, we restrict attention to the case $\xi > -1$. If $\Xx$ and $H$ are compact metric spaces and if $\bth : \Xx \times H \to (0, \infty) \times (-1, \infty)$ is a continuous link function, then in the set-up of \cref{ass:paramodel}, \cref{cor:cmle:strong} yields the strong consistency of the conditional maximum likelihood estimator for the true regression parameter $\eta_0$ within $H$.
\end{example}

\subsection{Fitting heavy-tailed regression models to block maxima}
\label{subsec:block-maxima}

We apply the triangular set-up of \cref{sec:consistency-covariates} to show the consistency of the maximum likelihood estimator in a regression model for the Fréchet distribution fitted to heavy-tailed block maxima in the presence of covariates. The distribution of the block maxima depends on the block size and becomes Fréchet only in the limit as the block size tends to infinity. This feature motivates the dependence of the Markov kernel on $n$ in \cref{ass:dgpZ}. This section's main result, presented in \cref{thm:Frechet:strong}, can be regarded as a version of the main result in \cite{Dom15}, restricted to the heavy-tailed case but extended to the conditional setting.

Let $F_0$ be a cdf on $\reals$ in the max-domain of attraction of the Fréchet distribution with shape parameter $\alpha \in (0,\infty)$: there exists a sequence $a_r > 0$ such that for all $y \in \reals$ we have
\begin{equation}
    \label{eq:F0DoA}
    \lim_{r\to\infty} F_0^r(a_r y) = \Phi_\alpha(y) =
    \begin{dcases}
        \exp\rbr{-y^{-\alpha}}, & \text{if $y > 0$,} \\
        0, & \text{if $y \le 0$.}
    \end{dcases}
\end{equation}
For all $x$ in some separable metric space $\Xx$, let $F_x$ be a cdf on $\reals$ such that the map $(x, y) \mapsto F_x(y)$ is Borel measurable. Assume that there exists $c : \Xx \to (0, \infty)$ 
such that
\begin{equation}
    \label{eq:FxF0}
    \lim_{y\to\infty} \sup_{x \in \Xx} \abr{\frac{\log F_x(y)}{\log F_0(y)} - c(x)} = 0.
\end{equation}
In words, the tail $1 - F_x(y) \sim - \log F_x(y)$ can be approximated uniformly well by the tail of $F_0$, up to a covariate-dependent multiplicative constant $c(x)$. The assumption mirrors that of heteroscedastic extremes in \cite{EinDeHZho16}.

Similar as in \cite{Dom15}, we observe $n$ independent block maxima $M_{n,1},\ldots,M_{n,n}$ together with $n$ covariates $x_{n,1},\ldots,x_{n,n} \in \Xx$. 
The block maxima take the form
\begin{equation}
\label{eq:Mnixi}
    M_{n,i} = \max_{t \in [r_n]} \xi_{n,i;t},
\end{equation}
where $r_n \in \nat$ is the block size and $\xi_{n,i;1}, \dots, \xi_{n,i;r_n}$ is a potentially unobserved independent random sample drawn from $F_{x_{n,i}}$. The cdf of the rescaled block maximum $M_{n,i} / a_{r_n}$ is thus
\[
    \pr\rbr{M_{n,i}/a_{r_n} \le y} = F_{x_{n,i}}^{r_n}(a_{r_n} y).
\]
Asymptotically, this distribution is Fréchet with covariate-dependent scale parameter.

\begin{lemma}
    \label{lem:avoid-uniform-rv}
    Under \cref{eq:F0DoA,eq:FxF0}, we have, for all $y > 0$,
    \begin{equation}
        \label{eq:frechet-doa-uniform}
        \lim_{r\to\infty} \sup_{x \in \Xx} \abr{
            F_x^{r}(a_{r}y) - \Phi_\alpha\bigl(y/\sigma(x)\bigr)
        } = 0
        \qquad \text{with} \quad
        \sigma(x) = c(x)^{1/\alpha}.
    \end{equation}
\end{lemma}

\begin{proof}
    Fix $y > 0$. For $x \in \Xx$ and $r \in \nat$, we have
    \[
        F_x^{r}(a_r y)
        = \exp \cbr{ r \log F_0(a_r y) \cdot \frac{\log F_x(a_r y)}{\log F_0(a_r y)}}
        = \rbr{F_0^{r}(a_r y)}^{\frac{\log F_x(a_r y)}{\log F_0(a_r y)}}.
    \]
    By \cref{eq:FxF0}, the big exponent converges to $c(x)$ uniformly in $x \in \Xx$, so by \cref{eq:F0DoA}, the whole expression converges to
    $
        [\exp(-y^{-\alpha})]^{c(x)}
        = \exp \rbr{-c(x) y^{-\alpha}},
    $
    uniformly in $x \in \Xx$; this is~\cref{eq:frechet-doa-uniform}. 
\end{proof}

\begin{remark}
    \label{rk:uniform-rv}
    A converse statement to \cref{lem:avoid-uniform-rv} holds as well in the sense that if $c(x)$ is bounded away from zero and infinity, then \eqref{eq:frechet-doa-uniform} implies the existence of $F_0$ such that \cref{eq:F0DoA,eq:FxF0} hold. To see this, simply put $F_0 = F_{x_0}$ for some arbitrary $x_0 \in \Xx$. \cref{eq:F0DoA} is then immediate, while \cref{eq:FxF0} requires some careful analysis, of which we omit the details for brevity. 
\end{remark}

We now suppose that the scale function $\sigma(x)$ in \cref{eq:frechet-doa-uniform} can be described by a parameter $\beta \in B$. A typical case is the loglinear model 
$
    \sigma_\beta(x) = \exp \bigl(\beta^\top x\bigr)
$
where $\Xx, B \subset \reals^d$. The aim is to estimate the triple $(\alpha, \beta, a_{r_n})$ on the basis of the observed covariate-maxima pairs $(x_{n,i},M_{n,i})$ for $i \in [n]$. Under \cref{eq:Mnixi}, the distribution of $M_{n,i}$ is approximately Fréchet with shape parameter $\alpha$ and scale parameter $a_{r_n} \sigma_\beta(x_{n,i})$, so we can hope to estimate the parameters consistently by (conditional) maximum likelihood.

The probability density function of the Fréchet distribution $\Phi_\alpha$ in \eqref{eq:F0DoA} is
\[
    p_{1,\alpha}(y) = 
    \begin{dcases} 
        \alpha y^{-\alpha-1} \exp \rbr{-y^{-\alpha}} & \text{if $y > 0$,} \\
        0 & \text{if $y \le 0$.}
    \end{dcases}
\]
The density of the Fréchet distribution with scale-shape parameter vector $(\tau, \alpha) \in (0, \infty)^2$ is thus
\[
    p_{\tau,\alpha}(y) 
    = \tau^{-1} p_{1,\alpha}(y/\tau)
    = \frac{\alpha}{\tau} (y/\tau)^{-\alpha-1} \exp \rbr{-(y/\tau)^{-\alpha}} 
\]
The rescaled loglikelihood based on the observed covariate-maxima pairs is then
\begin{equation}
    \label{eq:Fr:ll}
    L_n(\tau, \beta, \alpha) 
    := \frac{1}{n} \sum_{i=1}^n \log p_{\tau \sigma_\beta(x_{n,i}), \alpha}(M_{n,i}).
\end{equation}
Let $\alpha_0 \in (0, \infty)$ and $\beta_0 \in B$ denote the true shape and regression parameters, respectively, in the sense that for every $y \in (0,\infty)$ we have, as in \cref{lem:avoid-uniform-rv},
\[
    \lim_{n\to\infty}
    \max_{i\in[n]} \abr{
        \pr\rbr{M_{n,i}/a_{r_n} \le y}
        -
        \Phi_{\alpha_0}\rbr{y/\sigma_{\beta_0}(x_{n,i})}
    } = 0.
\]

\begin{theorem}
    \label{thm:Frechet:strong}
    Let $A \subset (0, \infty)$ be compact and let $B, \Xx$ be compact metric spaces. Let $(M_{n,i}: n \in \nat, i \in [n])$ be as in \cref{eq:Mnixi} with distributions $F_x$ as in \cref{eq:F0DoA,eq:FxF0}, where $\alpha_0 \in A$ denotes the true parameter and where $c(x) = \sigma_{\beta_0}(x)^{\alpha_0}$ for some continuous map $(x, \beta) \mapsto \sigma_{\beta}(x) \in (0,\infty)$ 
    and some $\beta_0 \in B$. Assume that $\log n = o(r_n)$, that $n^{-1} \sum_{i=1}^n x_{n,i} \wto P_X$, and that $\beta_0$ is identifiable in the sense that
    \begin{equation}
        \label{eq:Fr:ident}
        \forall (\beta, \gamma) \in \rbr{B \times (0, \infty)} \setminus \cbr{(\beta_0, 1)}:
        \qquad 
        P_X \rbr{\cbr{x \in \Xx : \gamma \sigma_\beta(x) \ne \sigma_{\beta_0}(x)}} > 0.
    \end{equation}
    Let $0 < \gamma_- < 1 < \gamma_+ < \infty$. Then any random sequence $\hetan = (\hat{\alpha}_n, \hat{\beta}_n, \hat{\tau}_n)$ in $A \times B \times [\gamma_- a_{r_n}, \gamma_+ a_{r_n}]$ such that $L_n(\hetan) \ge L_n(\alpha_0, \beta_0, a_{r_n}) + \oh_{a.s.}(1)$ satisfies
    \[
        \rbr{\hat{\alpha}_n, \hat{\beta}_n, \hat{\tau}_n/a_{r_n}}
        \to \rbr{\alpha_0, \beta_0, 1}, \qquad n \to \infty, \quad \text{a.s.}
    \]
\end{theorem}

\begin{remark}[Block size]
    The condition $\log n = o(r_n)$ requires that the block size $r_n$ grows faster than logarithmically with the number of block maxima $n$. It serves to ensure that with large probability, all block maxima are bounded away from zero, see \cref{lem:frechet-omega-0-probability-1}. The same condition has been imposed in \cite[Theorem 2]{Dom15}.
\end{remark}

\begin{remark}[Scaling sequence]
    The scaling sequence $a_{r_n}$ is unknown, and so is the search interval $[\gamma_- a_{r_n}, \gamma_+ a_{r_n}]$ for the overall scale parameter $\tau$ of the conditional maximum likelihood estimator $\hetan$. This issue can be solved by taking $\gamma_-$ and $\gamma_+$ sufficiently small and large, respectively, and replacing $a_{r_n}$ by the median (or any other fixed quantile) of $M_{n,1},\ldots,M_{n,n}$. This median is of the form $\rbr{\kappa + \oh_{a.s.}(1)} a_{r_n}$ with $\kappa > 0$ a model-dependent constant, more precisely the corresponding quantile of the $\sigma_\beta(X)$-scale mixture of the Fréchet distribution with shape parameter $\alpha_0$, where $X$ has distribution $P_X$.
\end{remark}

\section{Argmax theorems without uniform convergence}
\label{sec:argmax}

The Argmax continuous mapping theorem in \citet[Theorem~3.2.2]{vaaWel23} states conditions under which the argmax of a sequence of random criterion functions converges weakly to the argmax of the weak limit, in an appropriate function space, of those criterion functions. In Theorem~2.2 in \cite{lachout2005strong}, such functional convergence is not required, and an upper bound on suprema over certain sets suffices. The two theorems below provide slightly modified versions of the latter theorem, versions that are at the basis of our \cref{thm:strongZ} and \cref{thm:weakZ}.

Let $(\Ee, \Ec)$ be a measurable space.
Let $\Pemp$, for $n \in \nat$, be a sequence of random probability measures on $\Ee$ 
and let $P$ be a non-random probability measure on $\Ee$, respectively; we think of $\Pemp$ as being close to $P$ for large $n$.
Write integrals of extended real-valued functions $f: \Ee \to [-\infty, \infty]$ as $\Pemp f$ and $P f$.
Let $(H, d_H)$ be a compact metric space and let $m : H \times \Ee \to [-\infty, \infty)$ be a measurable function. We frequently write $m_\eta(z) = m(\eta, z)$ for $\eta \in H$ and $z \in \Ee$.
The integrals $M(\eta) = P m_\eta$ and $M_n(\eta) = \Pemp m_\eta$ are assumed to exist in $[-\infty, \infty]$ for all $\eta \in H$; condition~\ref{L-identifiable} below will in fact imply that $M(\eta)$ is not equal to $\infty$. Recall that $m_B(z) = \sup_{\eta \in B} m_\eta(z)$ for $B \subset H$ and $z \in E$,
and that $B(\eta, \delta)$ and $\clB(\eta,\delta)$ are the open and closed balls, respectively, in $H$ with center $\eta \in H$ and radius $\delta > 0$. 

\begin{theorem}   
\label{thm:lachout-modified}
Assume the above set-up with the following conditions:
\begin{enumerate}[label=(A\arabic*)]
\item \label{L-identifiable} 
The function $M:H \to [-\infty, \infty]$ has a unique point of finite maximum $\eta_0 \in H$, that is, 
\[
    -\infty \le M(\eta) < M(\eta_0)<\infty \qquad \text{ for all } \eta \in H \setminus \{ \eta_0\}.
\]
\item \label{L-semicontinuous}        
For every $\eta \in H$, there exists $V_\eta \in \Ec$ with $P(V_\eta) = 0$ such that for all $z \in \Ee \setminus V_\eta$, the function $\bar{\eta} \mapsto m_{\bar{\eta}}(z)$ is upper semicontinuous at $\eta$.
\item \label{L-bounded-integrals-convergence-suprema}
For all $\eta \in H \setminus \cbr{\eta_0}$ there exists $\delta > 0$ such that $P m_{\clB(\eta,\delta)} < \infty$ and such that, for all $\rho \in (0, \delta]$, we have 
\[
    \limsup_{n\to\infty} \Pemp m_{\clB(\eta,\rho)} \le P m_{\clB(\eta,\rho)},
    \qquad \text{a.s.}
\]
\item \label{L-convergence-eta0} We have
        \[
            \liminf_{n\to\infty} M_n(\eta_0) \ge M(\eta_0), \qquad \text{a.s.}
        \]
\end{enumerate}
Then every estimator $\hetan$ satisfying $M_n(\hetan) \ge M_n(\eta_0) - \oh_{a.s.}(1)$ is strongly consistent for the estimation of $\eta_0$, that is, $\lim_{n\to\infty} \hetan = \eta_0$ a.s. 
\end{theorem}

The following result is a variation of the previous theorem for weak rather than strong consistency.

\begin{theorem}
\label{thm:lachout-modified:weak}
Assume the above set-up with Assumptions~\ref{L-identifiable} and \ref{L-semicontinuous} as in \cref{thm:lachout-modified}, and with the following two assumptions:
\begin{enumerate}[label=(A\arabic*')]
    \setcounter{enumi}{2}
    \item \label{L-convergence-suprema-P} 
    For all $\eta \in H \setminus \cbr{\eta_0}$ there exists $\delta > 0$ such that $P m_{\clB(\eta,\delta)} < \infty$ and such that, for all $\rho \in (0, \delta]$, we have 
\[
    \Pemp m_{\clB(\eta,\rho)} \le P m_{\clB(\eta,\rho)} + \oh_{\pr}(1),
         \qquad n \to \infty.
\]
    \item \label{L-convergence-eta0-P} We have
    \[
        M_n(\eta_0) \ge M(\eta_0) - \oh_{\pr}(1),
        \qquad n \to \infty.
    \]
\end{enumerate}
Then every estimator $\hetan$ satisfying $M_n(\hetan) \ge M_n(\eta_0) - \oh_{\pr}(1)$ is weakly consistent for the estimation of $\eta_0$, that is, $\lim_{n\to\infty} \pr\sbr{d_H(\hetan,\eta_0) \ge \delta_0} = 0$ for every $\delta_0 > 0$.
\end{theorem}

\begin{remark}[Comparison with Theorem~2.2 in \cite{lachout2005strong}]
As mentioned above, \cref{thm:lachout-modified} is a slightly modified version of Theorem~2.2 in \cite{lachout2005strong}. Translating their notation to our setup, the most important differences are as follows. While they require $m_\eta$ to be real-valued, we permit $m_\eta(z) = -\infty$, thus accommodating conditional maximum likelihood estimators when the support is parameter-dependent. Their estimator $\hat \eta_n$ must satisfy $M_n(\hat \eta_n) \ge \sup_{\eta \in H} M_n(\eta) - \oh_{a.s.}(1)$, while we only require $M_n(\hat \eta_n) \ge M_n(\eta_0) - \oh_{a.s.}(1)$. Finally, they allow for the case where the criterion function $M$ may have multiple maxima, and they show almost sure convergence of the estimator to the set of maximizers, as in Theorem~5.14 in \cite{van1998asymptotic}. The latter conclusion coincides with ours in case the set of maximizers is a singleton, as required in condition~\ref{L-identifiable}.
\end{remark}

\section{Auxiliary results and proofs}
\label{sec:proofs}

\subsection{Proofs for Section~\ref{sec:consistency-general}}

Let $\Pemp = n^{-1} \sum_{i=1}^n \delta_{Z_{n,i}}$ denote the empirical distribution of the $n$th row of variables in \cref{ass:dgpZ}.

\begin{lemma}[Strong law of large numbers]
    \label{lem:PniLLNZ}
    Under \cref{ass:dgpZ}, if $f : \Zz \to \reals$ is Borel measurable and continuous $P_{Z}$-almost everywhere, and if the family $\{f(Z_{n,i}): n \ge n_0, i \in [n]\}$ is $L^2$-stochastically dominated for some $n_0 \in \nat$, then $P_{Z} |f| < \infty$ and $\Pemp f \to P_{Z} f$ almost surely as $n\to\infty$.
\end{lemma}

\begin{proof}[Proof of \cref{lem:PniLLNZ}]
    First we show $P_{Z} |f| < \infty$. Let $W$ be a nonnegative random variable such that $\expec[W^2] < \infty$ and such that, for some $n_0 \in \nat$,
    \[ 
        \forall n \ge n_0, i \in [n], t \ge 0: \qquad 
        P_{n,i}\cbr{|f| > t} \le \pr(W > t). 
    \]
    Then $P_{n,i} f^2 \le \expec[W^2]$ and thus also $P_n f^2 \le \expec[W^2]$ where $P_n = n^{-1} \sum_{i=1}^n P_{n,i}$. It follows that $P_Z |f| = \lim_{k\to\infty} P_Z (|f| \wedge k)$ is finite, since 
    \[ 
        P_Z (|f| \wedge k) 
        = \lim_{n\to\infty} P_n (|f| \wedge k) 
        \le \lim_{n\to\infty} (P_n f^2)^{1/2} 
        \le (\expec[W^2])^{1/2}. 
    \]
    
    Next we show the almost sure convergence of $\Pemp f$. We have 
    \[ 
        \expec[\Pemp f] 
        = \frac{1}{n} \sum_{i=1}^n P_{n,i} f 
        = P_n f 
    \]
    and thus
    \[
        \Pemp f - \expec[\Pemp f]
        = \frac{1}{n} \sum_{i=1}^n \epsilon_{n,i}
        \qquad \text{where} \quad
        \epsilon_{n,i} = f(Z_{n,i}) - P_{n,i} f.
    \]
    The random variables $\epsilon_{n,i}$ are rowwise independent and the array $\cbr{\epsilon_{n,i} : n \ge n_0, i \in [n]}$ is $L^2$-stochastically dominated by $W + (\expec[W^2])^{1/2}$; indeed, $\abr{\epsilon_{n,i}} \le \abr{f(Z_{n,i})} + \abr{P_{n,i} f}$ and $\abr{P_{n,i} f} \le (P_{n,i} f^2)^{1/2} \le (\expec[W^2])^{1/2}$ for all $n \ge n_0$ and $i \in [n]$. By Theorem~2 in \cite{HuMorTay89}, we find $\Pemp f - \expec[\Pemp f] \to 0$ almost surely as $n \to \infty$.

    It remains to show that $P_n f \to P_{Z} f$ as $n \to \infty$. Since $P_n \wto P_{Z}$ as $n\to\infty$ and since $f$ is continuous $P_{Z}$-almost everywhere, the continuous mapping theorem \citep[Theorem~1.11.1]{vaaWel23} implies that $f_{\#} P_n \wto f_{\#} P_{Z}$ as $n \to \infty$; in words, the distribution of $f$ under $P_n$ converges to the one of $f$ under $P_{Z}$. Moreover, by $L^2$-stochastic dominance, $\sup_{n \ge n_0} P_n f^2 \le \sup_{n \ge n_0, i \in [n]} P_{n,i} f^2 < \infty$. By Example~2.21 in \cite{van1998asymptotic}, the first moment of $f$ under $P_n$ converges to the first moment of $f$ under $P_{Z}$.
\end{proof}

\begin{corollary}[Varadarajan's theorem for triangular arrays, \citealp{Var58}]
    \label{cor:VaradarajanZ}
    Under \cref{ass:dgpZ}, we have $\pr \rbr{\Pemp \wto P_{Z} \text{ as } n\to\infty} = 1$.
\end{corollary}

\begin{proof}[Proof of \cref{cor:VaradarajanZ}] 
    By \cref{lem:vara_sep_metr_wconv}, it is sufficient to show that $\Pemp f \to P_{Z} f$ almost surely as $n\to\infty$ for every bounded and continuous function $f : \Zz \to \reals$. But this is an immediate consequence of \cref{lem:PniLLNZ}.
\end{proof}

\begin{lemma}\label{lem:vara_sep_metr_wconv}
    Let $P, P_1, P_2, \ldots$ be probability measures on a metric space $\Mm$ equipped with its Borel $\sigma$-field. If $\Mm$ is separable, there exists a countable set $\mathcal{D}\subset \operatorname{Lip}_b(\Mm)$ such that $P_n g \to P g$ for all $g \in \mathcal{D}$ already implies $P_n \wto P$.
\end{lemma}

\begin{proof} 
    Fix $f \in \operatorname{Lip}_b(\Mm)$. By the Portmanteau theorem, it is sufficient to show that $P_n f \to Pf$.
    By the proof of Theorem 3.1 in \cite{varadarajan1958weak}, we may assume that $\Mm$ is totally bounded. Arguments in the proof of Theorem 11.4.1 in \cite{Dud02} yield the separability of $(\operatorname{Lip}_b(\Mm), \| \cdot \|_\infty)$. 
    Thus, there exists a countable and dense subset $\mathcal{D} \subset (\operatorname{Lip}_b(\Mm), \|\cdot \|_\infty)$. Let $\varepsilon > 0$ and choose a $g \in \mathcal{D}$ with $\|f - g\|_\infty < \varepsilon/3$. Further, by assumption there exists $n_0$ such that for all $n \geq n_0$ we have $|P_n g - P g | \leq \varepsilon/3$. Hence, 
    \[
        |P_n f - P f|
        \leq |P_n f - P_n g| + |P_ng - Pg | + |Pg - Pf|
        \leq \varepsilon
    \]
    for all $n \geq n_0$ which implies the assertion. 
\end{proof}

\begin{lemma}[Portmanteau for unbounded upper semicontinuous functions] 
    \label{lem:portmanteauuscZ}
    Under \cref{ass:dgpZ}, if $f : \Zz \to \reals$ is upper semicontinuous and quasi-integrable with respect to $P_{Z}$, and if the family $\{(f^+(Z_{n,i}): n \ge n_0, i \in [n]\}$ is $L^2$-stochastically dominated for some $n_0 \in \nat$, then 
    \[
    \pr\Big[\limsup_{n \to \infty}\Pemp f \le P_{Z} f \Big] = 1.
    \]
\end{lemma}

\begin{proof}[Proof of \cref{lem:portmanteauuscZ}] 
    Nothing has to be shown if $P_{Z}f=\infty$, so assume that $P_{Z}f<\infty$.
    Write $f=g_1 + g_2$ where $g_1 = \max(f,0) = f^+ \ge 0$ and $g_2 = \min(f,0) \le 0$, and note that both $g_1$ and $g_2$ are upper semicontinuous. Since $\Pemp f = \Pemp g_1 + \Pemp g_2$ and $P_Z f = P_Z g_1 + P_Z g_2$, it is sufficient to show that
    \begin{align}
    \label{eq:gjZ}
    \pr\Big[ \limsup_{n \to \infty} \Pemp g_j \le P_Z g_j \Big]= 1
    \end{align}
    for $j=1$ and $j=2$ separately. 

    For $j=2$, the claim in \cref{eq:gjZ} follows immediately from \cref{cor:VaradarajanZ} and item~(v) in the Portmanteau Theorem~1.3.4 in \citet{vaaWel23}, as $g_2$ is upper semicontinuous and bounded above by 0.

    For $j=1$, we may repeat the arguments in the second paragraph of the proof of \cref{lem:PniLLNZ} to find that $\Pemp g_1 - \expec[ \Pemp g_1 ]= \oh(1)$ almost surely as $n \to \infty$. Since $\expec[\Pemp g_1] = P_n g_1$ with $P_n = n^{-1} \sum_{i=1}^n P_{n,i}$, it remains to show that 
    \begin{equation}
    \label{eq:Png1PZg1}
        \limsup_{n \to \infty} P_n g_1 \le P_{Z} g_1.
    \end{equation}
    By assumption, the family $\cbr{(g_1)_{\#} P_{n,i}: n\ge n_0,i\in[n]}$ is $L^2$-stochastically dominated, and therefore so is the family $\cbr{(g_1)_{\#} P_n: n\ge n_0}$. As a consequence, $\sup_{n\ge n_0} P_n [g_1^2] < \infty$, so that $g_1$ is uniformly integrable under $P_n$, i.e., $\lim_{k\to\infty} \sup_{n\ge n_0} P_n [g_1 \1_{\{g_1>k\}}] = 0$. But then we have for all $k > 0$ the inequalities
    \begin{align*} 
        \limsup_{n\to\infty} P_n g_1 
        &\le \limsup_{n\to\infty} P_n [\min(g_1, k)] + \sup_
        {n \ge n_0} P_n [(g_1-k)^+] \\
        &\le P_{Z} [\min(g_1, k)] + \sup_{n \ge n_0} P_n [g_1 \1_{\{g_1>k\}}], 
    \end{align*}
    where, in the last inequality, we used again item~(v) in Theorem~1.3.4 in \citet{vaaWel23}.
    Letting $k \to \infty$ yields \cref{eq:Png1PZg1}, as required. 
\end{proof}

The following permanence property of upper semicontinuous functions can be easily proven using a subsequence argument but is also part of a special case of Bergman's maximum theorem; see for instance Lemma~17.30 in \cite{aliprantis2006infinite}.

\begin{lemma}[Upper semicontinuity]
\label{lem:usc}
	Let $\Zz$ and $H$ be metric spaces and let $f : H \times \Zz \to [-\infty,\infty]$ be upper semicontinuous.
    If $K \subset H$ is non-empty and compact, then the function $f_K : \Zz \to [-\infty,\infty]$ defined by
	\[
		f_K(z) = \sup_{\eta \in K} f(\eta, z), 
		\qquad z \in \Zz,
	\]
	is upper semicontinuous.
\end{lemma}

\begin{proof}[Proof of \cref{thm:strongZ}] 
    The result follows from \cref{thm:lachout-modified}, with $\Ee = \Zz$, $P = P_Z$ and $\Pemp = n^{-1} \sum_{i=1}^n \delta_{Z_{n,i}}$. We only need to check Conditions~\ref{L-identifiable}--\ref{L-convergence-eta0}.
    \begin{itemize}
        \item Assumption~\ref{L-identifiable} is part of \cref{ass:trueZ}.
        \item Assumption~\ref{L-semicontinuous} follows from \cref{ass:critZ}, since slicing an upper semicontinuous function on a product space to a single variable produces an upper semicontinuous function again.
        \item Finiteness of the integral in Assumption~\ref{L-bounded-integrals-convergence-suprema} is a consequence of the first part of \cref{ass:stochdomZ}(ii). The inequality involving the limsup in Assumption~\ref{L-bounded-integrals-convergence-suprema} follows from \cref{lem:portmanteauuscZ} applied to $m_{\clB(\eta,\rho)}$: the latter function is upper semicontinuous by \cref{ass:critZ} and \cref{lem:usc}, and its positive part is $L^2$-stochastically dominated by \cref{ass:stochdomZ}(ii).
        \item Assumption~\ref{L-convergence-eta0} is a consequence of \cref{lem:PniLLNZ} applied to $f = m_{\eta_0}$ in combination with \cref{ass:trueZ} and \cref{ass:stochdomZ}(i).
        \qedhere
    \end{itemize}
\end{proof}

\begin{proposition}[Weak law of large numbers for triangular arrays]
    \label{prop:WLLNtri}
    For some $n_0 \in \nat$, consider a triangular array $\cbr{X_{ni} : n \ge n_0, i \in [n]}$ of rowwise independent random variables. If the family is uniformly integrable, see \cref{eq:def:uniformly-integrable}, then
    \[
        \frac{1}{n} \sum_{i=1}^n \rbr{X_{ni} - \expec[X_{ni}]}
        \to 0 \qquad \text{in probability, as $n \to \infty$.}
    \]
\end{proposition}

\cref{prop:WLLNtri} is a direct consequence of the main theorem in \cite{gut1992weak}; uniform integrability in the Cesàro sense is already sufficient. According to \citet[p.~215]{chandra1992cesaro}, uniform integrability is \emph{not} enough to conclude the \emph{strong} law of large numbers. 

\begin{lemma}
    \label{lem:portmanteauuscZ-P}
    Under \cref{ass:dgpZ}, if $f : \Zz \to \reals$ is upper semicontinuous and quasi-integrable with respect to $P_Z$, and if the family $\cbr{f(Z_{n,i})^+ : n \ge n_0, i \in [n]}$ is uniformly integrable, then $\Pemp f \le P_Z f + \oh_{\pr}(1)$ as $n \to \infty$.
\end{lemma}

\begin{proof}[Proof of \cref{lem:portmanteauuscZ-P}] 
    If $P_Z f = \infty$, there is nothing to prove, so assume $P_Z f < \infty$.
    As in the proof of \cref{lem:portmanteauuscZ}, we apply the same decomposition $f = g_1 + g_2$ with $g_1 = f^+ \ge 0$ and $g_2 = \min(f, 0) \le 0$. It is sufficient to show that $\Pemp g_j \le P_Z g_j + \oh_{\pr}(1)$ for $j \in \cbr{1,2}$.
    
    For $j = 2$, the desired inequality is even true almost surely, by the same argument as in the proof of \cref{lem:portmanteauuscZ}, thanks to the Pormanteau Theorem~1.3.4 in \cite{vaaWel23}: we have $\Pemp \wto P_Z$ almost surely by \cref{cor:VaradarajanZ} (which only relies on \cref{ass:dgpZ}), while $g_2$ is upper semicontinuous and bounded above (by $0$).

    For $j = 1$, we apply \cref{prop:WLLNtri} to the uniformly integrable triangular array $X_{ni} = g_1(Z_{n,i})$ for $n \ge n_0$ and $i \in [n]$: we have
    \begin{align*}
        \Pemp g_1 
        &= \frac{1}{n} \sum_{i=1}^n g_1(Z_{n,i}) 
        = \frac{1}{n} \sum_{i=1}^n X_{ni}, \\
        \frac{1}{n} \sum_{i=1}^n \expec[X_{ni}]
        &= \frac{1}{n} \sum_{i=1}^n P_{n,i} g_1
        = P_n g_1.
    \end{align*}
    As the difference between the two quantities converges to zero in probability thanks to \cref{prop:WLLNtri}, it remains to show that $\limsup_{n\to\infty}P_n g_1 \le P_Z g_1$. But this follows from the fact that $P_n \wto P$ (\cref{ass:dgpZ}) and the upper semicontinuity of $g_1$, by an application of the Portmanteau Theorem~1.3.4 in \cite{vaaWel23}.
\end{proof}

\begin{proof}[Proof of \cref{thm:weakZ}]
    This time, we apply \cref{thm:lachout-modified:weak}. 

    Assumptions~\ref{L-identifiable} and~\ref{L-semicontinuous} are verified by the same arguments as in the proof of \cref{thm:strongZ}.

    The integrability condition in Assumption~\ref{L-convergence-suprema-P} is a consequence of the first part of \cref{ass:stochdomZ}(ii). The inequality involving the convergence in probability in Assumption~\ref{L-convergence-suprema-P} follows from \cref{lem:portmanteauuscZ-P} applied to $f = m_{\clB(\eta,\delta)}$: this function is upper semicontinuous by \cref{ass:critZ} and \cref{lem:usc}, and its positive part is uniformly integrable with respect to the distributions $P_{n,i}$ by \cref{ass:uiZ}(ii).
    
    Assumption~\ref{L-convergence-eta0-P} is a consequence of \cref{prop:WLLNtri} applied to $X_{ni} = m_{\eta_0}(Z_{n,i})$, which are uniformly integrable by \cref{ass:uiZ}(i): we have
    \[
        M_n(\eta_0) 
        = \frac{1}{n} \sum_{i=1}^n m_{\eta_0}(Z_{n,i}) 
        = \frac{1}{n} \sum_{i=1}^n X_{n,i}
    \]
    with expectation
    \[
        \frac{1}{n} \sum_{i=1}^n \expec[X_{n,i}]
        = \frac{1}{n} \sum_{i=1}^n P_{n,i}m_{\eta_0}
        = P_n m_{\eta_0}.
    \]
    By Theorem~2.20 in \citet{van1998asymptotic}, $P_n m_{\eta_0}$ converges to $M(\eta_0) = P_Z m_{\eta_0}$ since $P_n \wto P_Z$ (\cref{ass:dgpZ}), $m_{\eta_0}$ is continuous $P_Z$-almost everywhere (\cref{ass:trueZ}) and $(m_{\eta_0})_{\#} P_n$ is uniformly integrable (as a consequence of \cref{ass:uiZ}(i) and the identity $P_n = n^{-1} \sum_{k=1}^n P_{n,k}$).
\end{proof}

\subsection{Proofs for Section~\ref{sec:consistency-covariates}}
\label{sec:consistency-covariates-proofs}

\begin{proof}[Proof of \cref{lem:from-conditional-to-triangular}] 
    We only need to show that $P_n \wto P_{X,Y}$, where $P_n=n^{-1} \sum_{i=1}^n P_{n,i}$. By \cref{lem:vara_sep_metr_wconv}, it is sufficient to show that $P_n f \to P_{X, Y} f$ as $n\to\infty$, for every bounded and uniformly continuous function $f : \XY \to \reals$. For that purpose, let 
    \begin{align*}
        \mu_n &= \frac{1}{n} \sum_{i=1}^n \delta_{x_{n,i}}, &
         h_n(x) &= \int_\Yy f(x,y) \,Q_n(x, \diff y), & 
        h(x) &= \int_\Yy f(x,y) \,Q(x, \diff y),
    \end{align*}
    so that $P_n f = \mu_n h_n$ and $P_{X,Y} f = P_X h$.
    The proof will be finished once we show $\mu_n h_n \to P_X h$ as $n\to\infty$. Note that, by \cref{ass:dgpZ-conditional}(a), we have $\mu_n \wto P_X$ as $n\to\infty$.

    First, we have that $h_n - h$ converges continuously to zero. Indeed let $x_n \to \bar x \in \Xx$ and $\varepsilon > 0$. By uniform continuity of $f$ and \cref{ass:dgpZ-conditional}(b), we may choose $n_0 \in \nat$ such that for all $n \geq n_0$ we have $\|f(x_n, \point) - f(\bar x, \point)\|_\infty < \varepsilon / 4$ and  $|\{Q_n(x_n, \point) - Q(\bar x, \point)\}[f(\bar x, \point)]| \vee |\{Q(x_n, \point) - Q(\bar x, \point)\}[f(\bar x, \point)]| < \varepsilon/4$. Then, for all $n \geq n_0$,
    \begin{align*}
        \lefteqn{ |h_n(x_n) - h(x_n)| } \\
        &\le \int_\Yy |f(x_n, y) - f(\bar x, y)|\, Q_n(x_n, \diff y) 
            + \int_\Yy |f(x_n, y) - f(\bar x, y)| \, Q(x_n, \diff y)\\
        &\hspace{1cm}
        + \Big | \int_\Yy f(\bar x, y) \, \{Q_n(x_n, \diff y) - Q(\bar x, \diff y)\} \Big |
        + \Big | \int_\Yy f(\bar x, y) \, \{Q(x_n, \diff y) - Q(\bar x, \diff y)\} \Big |, 
    \end{align*}
    which is strictly smaller than $\varepsilon$.
    By the extended continuous mapping theorem [Theorem~1.11.1 in \citet{vaaWel23}], this gives $h_n(Z_n) - h(Z_n) \wto 0$ for $Z_n \sim \mu_n$, and since $|h_n| \vee |h| \leq |f|$ and since $f$ is bounded, the moments also converge to zero, that is, $\mu_n[h_n - h] \to 0$. Thus, we are left with verifying $\mu_n h \to P_X h$. By \cref{ass:dgpZ-conditional}(a) this reduces to checking that $h$ is bounded and continuous. The first part was already verified and the continuity follows by similar but simpler arguments as in the last display.
\end{proof}

\subsection{Proofs for Section~\ref{sec:applications}}
\label{sec:applications-proofs}

\begin{proof}[Proof of \cref{thm:scoring:strong}]
    We apply \cref{cor:conditional:strong} and check Assumptions~\ref{ass:critZ}, \ref{ass:trueZ}, \ref{ass:stochdomZ} and \ref{ass:dgpZ-conditional}, starting with the latter.
    
    Regarding \cref{ass:dgpZ-conditional}, the Markov kernel is $Q_n(x, \diff y) = P_{\bth(\eta_0, x)}(\diff y)$. This kernel does not depend on $n$ and is weakly continuous as a function of $x$ by \cref{ass:paramodel}.
    
    Regarding \cref{ass:critZ}, the compactness of $H$ was assumed from the start. By \cref{ass:scoring}, the objective function $(\eta, x, y) \mapsto m_\eta(x,y) = S(P_{\bth(x,\eta)},y)$ attains values in $[-\infty, \infty)$. Furthermore, it is upper semicontinuous as the link function $\bth$ is continuous and the map $(\theta,y) \mapsto S(P_{\theta},y)$ is upper semicontinuous by Assumptions~\ref{ass:paramodel} and~\ref{ass:scoring}, respectively.
    Finally, since $\{S(P_\theta,Y) : \theta \in \Theta, Y \sim P_\theta\}$ is $L^2$-stochastically dominated\footnote{Or uniformly integrable in the case of weak consistency.} by Assumption~\ref{ass:scoring-dominance}, we have $\sup_{\theta\in\Theta} |S(P_\theta,P_\theta)| < \infty$, which in turn implies that $m_\eta$ is $P_{X,Y}$-integrable, with $P_{X,Y}$ in \cref{eq:PXYeta0}, as $S(P_{\bth(x,\eta)},P_{\bth(x,\eta_0)}) \le S(P_{\bth(x,\eta_0)},P_{\bth(x,\eta_0)})$.
    
    Regarding \cref{ass:trueZ}, the continuity of $m_{\eta_0}$ is stipulated in \cref{ass:scoring}. As argued in the proof of \cref{ass:critZ}, $m_\eta$ is $P_{X,Y}$-integrable, which implies that $M(\eta_0) = \int_{\Xx} S(P_{\bth(x,\eta_0)},P_{\bth(x,\eta_0)}) \, P_X(\diff x) < \infty$. Moreover, since $S$ is a strictly proper scoring rule, the identifiability of the parametric model in \cref{ass:paramodel} implies that 
    \begin{multline*}
        P_X \rbr{\cbr{x \in \Xx : 
            S(P_{\bth(x,\eta)}, P_{\bth(x,\eta_0)}) 
            < S(P_{\bth(x,\eta_0)}, P_{\bth(x,\eta_0)})
        }} \\
        = P_X \rbr{\cbr{x \in \Xx : \bth(x,\eta) \ne \bth(x, \eta_0)}}
        > 0,
    \end{multline*}
    from which 
    \begin{align*}
        M(\eta) 
        &= \int_{\Xx} S(P_{\bth(x,\eta)},P_{\bth(x,\eta_0)}) \, P_X(\diff x) \\
        &< \int_{\Xx} S(P_{\bth(x,\eta_0)},P_{\bth(x,\eta_0)}) \, P_X(\diff x)
        =M(\eta_0).
    \end{align*}   
    
    Finally, \cref{ass:stochdomZ} follows from \cref{ass:scoring-dominance} upon setting $Z_{n,i}=(x_{n,i},Y_{n,i})$:
    \begin{itemize}[itemsep=0pt, leftmargin=0pt, itemindent=30pt]
    \item[(i)] 
    Since $Y_{n,i} \sim P_{\bth(x_{n,i}, \eta_0)}$, the family $\{m_{\eta_0}(Z_{n,i}) = S(P_{\bth(x_{n,i}, \eta_0)}, Y_{n,i})\}_{n \ge 1, i \in [n]}$ is contained in the family $\cbr{S(P_\theta, Y) : \theta \in \Theta, Y \sim P_\theta}$. The $L^2$-stochastic dominance of the latter in \cref{ass:scoring-dominance}(i) then implies the same property for the former, yielding \cref{ass:stochdomZ}(i).
    \item[(ii)] 
    Let $\eta \in H \setminus \cbr{\eta_0}$ and let $\delta > 0$ be such that the requirement in \cref{ass:scoring-dominance}(ii) holds. Let $\mathcal{S}_{\eta,\delta}$ be the $L^2$-stochastically dominated family in \cref{eq:Setadeltafamily}. Let $W$ be a nonnegative, square-integrable random variable whose survival function dominates that of all random variables in $\mathcal{S}_{\eta,\delta}$. For all $x \in \Xx$, we have
    \[
        \int_{\Yy} 
            \sup_{\bar{\eta} \in \clB(\eta,\delta)} 
            \bigl(S(P_{\bth(x,\bar{\eta})}, y)\bigr)^+ \, 
        P_{\bth(x,\eta_0)}(\diff y)
        \le \expec[W] < \infty,
    \]
    and thus, by definition of $P_Z = P_{X,Y}$ in \cref{eq:PXYeta0},
    \[
        P_Z \bigl(m_{\clB(\eta,\delta)}\bigr)^+
        =
        \int_{\Xx} \int_{\Yy} 
            \sup_{\bar{\eta} \in \clB(\eta,\delta)} 
            \bigl(S(P_{\bth(x,\bar{\eta})}, y)\bigr)^+ \, 
        P_{\bth(x,\eta_0)}(\diff y) \, P_X(\diff x)
        < \infty.
    \]
    It follows that $m_{\clB(\eta,\delta)}$ is $P_Z$-quasi-integrable and that $P_Z m_{\clB(\eta,\delta)} < +\infty$. Further, since $Y_{n,i} \sim P_{\bth(x_{n,i}, \eta_0)}$, the family 
    \[ 
        \Big\{
            \rbr{m_{\clB(\eta,\delta)}(Z_{n,i})}^+ = 
            \sup_{\bar{\eta} \in \clB(\eta,\delta)} 
                \rbr{S(P_{\bth(x_{n,i},\bar{\eta})}, Y_{n,i})}^+ : 
            n \ge 1, i \in [n] 
        \Big\} 
    \]
    is contained in $\mathcal{S}_{\eta,\delta}$ and is therefore $L^2$-stochastically dominated too.
    \qedhere
    \end{itemize}
\end{proof}

\begin{proof}[Proof of \cref{ex:energy-score}]
    We apply \cref{thm:scoring:strong} and need to check Assumptions~\ref{ass:scoring} and~\ref{ass:scoring-dominance}. Regarding \cref{ass:scoring} recall that, as noted above, $\ES$ is a strictly proper scoring rule on $\Pc_\beta(\Rd)$. Next, writing $m := \sup_{\theta \in \Theta} \int_{\Rd} \|y\|^\beta \, P_\theta(\diff y)$, which is finite by \cref{eq:energy-moments}, we have
    \begin{align*}
        \ES(P_\theta, y)
        &\le \frac{1}{2} \expec_{P_\theta \otimes P_\theta}[\|Y-Y'\|^\beta] 
        \le 2^{\beta-1} \expec_{P_\theta}[\|Y\|^\beta]
        = 2^{\beta-1} m
    \end{align*}
        for all $\theta \in \Theta$ and $y \in \Rd$. As a consequence, $x \mapsto \ES(P_{\bth(x,\eta)}, P_{\bth(x,\eta_0)})$ is quasi-integrable with respect to $P_X$, for every $\eta \in H$. Moreover, the map $(\theta, y) \mapsto \ES(P_\theta, y)$ is continuous by the uniform integrability of $y \mapsto \nbr{y}^\theta$ with respect to $(P_\theta : \theta \in \Theta)$. Since $(x, \eta) \mapsto \bth(x, \eta)$ is continuous, the map $(x, y) \mapsto \ES(P_{\bth(x,\eta_0)},y)$ is then continuous too. We have hence shown \cref{ass:scoring}.
Finally, \cref{ass:scoring-dominance} follows from the Markov inequality and the assumption in \cref{eq:energy-moments}.
\end{proof}

\begin{proof}[Proof of \cref{cor:cmle:strong}]
    We apply \cref{thm:scoring:strong}, and need to check Assumptions~\ref{ass:paramodel}, \ref{ass:scoring}, and \ref{ass:scoring-dominance} for $S(P,y)=\LogS(P, y)$ from \cref{eq:LogS}.
    \cref{ass:paramodel} is part of the conditions in \cref{ass:conditional-mle}. Next, regarding \cref{ass:scoring}, the fact that $S$ is strictly proper follows from \citet[Section~4.1]{gneiting2007strictly}. We have $S(P_\theta,y) = \log p_\theta(y) < \infty$ by \cref{ass:conditional-mle}(i), and the two continuity requirements follow from (i) and (ii).
    Finally, \cref{ass:scoring-dominance}(i) and (ii) exactly correspond to \cref{ass:conditional-mle}(iii) and (iv), respectively. 
\end{proof}

\begin{proof}[Proof of \cref{ex:cmle_gev}]
    The map $(\theta,y) \mapsto p_\theta(y) \in [0, \infty)$ is continuous upon imposing the restriction $\xi > -1$, i.e., on the domain $\rbr{\reals \times (0, \infty) \times (-1,\infty)} \times \reals$; a key point is that $-1-1/\xi > 0$ for $-1 < \xi < 0$. 
    Hence, the same is true for $(\eta, x, y) \mapsto m_\eta(x,y) = \log p_{\bth(x, \eta)}(y) \in [-\infty, \infty)$, by the continuity of the link function $\bth$ and the assumption on its image $\Theta$. This settles the continuity requirements in \cref{ass:conditional-mle}.
    
    As $\bth$ is continuous and $\Xx$ and $H$ are compact, the image $\Theta$ is compact too. Let $\Theta_j$ be the projection of $\Theta$ onto its $j$th coordinate ($j = 1,2,3$). We have
    \begin{align}
        \nonumber
        \sup_{\eta \in H, x \in \Xx, y \in \reals} p_{\bth(x,\eta)}(y)
        = \sup_{\theta \in \Theta, y \in \reals} p_{\theta}(y) 
        &= \sup_{\theta \in \Theta, y \in \reals} \frac{1}{\sigma} p_{0,1,\xi}\rbr{\frac{y-\mu}{\sigma}} \\
        \label{eq:GEVboundedabove}
        &\le \sup_{\sigma \in \Theta_2, \xi \in \Theta_3, z \in \reals} 
        \frac{1}{\sigma} p_{0,1,\xi}(z)
        < \infty,
    \end{align}
     where we applied Proposition~1 in \cite{Dom15} in the last step, upon noting that $\Theta_2$ and $\Theta_3$ are compact subsets of $(0, \infty)$ and $(-1,\infty)$, respectively. We find that $m_\eta(x,y) = S(P_{\bth(x,\eta)},y) = \log p_{\bth(x,\eta)}(y)$ is bounded above uniformly in $(\eta,x,y)$. This deals with the $L^2$-stochastic dominance of the positive part in \cref{ass:conditional-mle}(iv).

    Regarding the $L^2$-stochastic dominance in \cref{ass:conditional-mle}(iii), it is sufficient to treat the negative part $(- \ell_\theta(y))^+$ of the loglikelihood and show that $\cbr{(- \ell_\theta(Y))^+ : \theta \in \Theta, Y \sim p_\theta}$ is $L^2$-stochastically dominated, since the positive part of the loglikelihood is already uniformly bounded above by \cref{eq:GEVboundedabove}. Let $\theta \in \Theta$ and $Y \sim p_\theta$. Write $Z = (Y - \mu)/\sigma$, with density $p_{0,1,\xi}$. Since $1 + \xi Z > 0$ a.s., it is sufficient to control $(- \ell_\theta(y))^+$ for $(\theta,y)$ such that $1 + \xi z > 0$ with $z = (y-\mu)/\sigma$, a restriction on $y$ (or $z$) which we will assume from now on without further mentioning. Write
    \[
        u(z) = u_\xi(z) = 
        \exp\rbr{ - \int_0^z \frac{1}{1 + \xi t} \diff t } =
        \begin{dcases}
            (1 + \xi z)^{-1/\xi} & \text{if $\xi \ne 0$,} \\
            e^{-z} & \text{if $\xi = 0$,}
        \end{dcases}
    \]
    with values in $(0, \infty)$. As $\xi > -1$ for all $\theta \in \Theta$, we have
    \begin{align*}
        \rbr{- \ell_\theta(y)}^+
        &= \rbr{\log(\sigma) + u(z) - (\xi + 1) \log(u(z))}^+ \\
        &\le \rbr{\log(\sigma)}^+ + u(z) + (\xi + 1) \rbr{-\log(u(z))}^+.
    \end{align*}
    Let $C_2 = C_2(\Theta) = \max_{\sigma \in \Theta_2} (\log\sigma)^+ \in [0, \infty)$ and $C_3 = C_3(\Theta) = \max_{\xi \in \Theta_3} (\xi +1) \in (0, \infty)$, with $\Theta_j$ as above. It follows that
    \[
        \rbr{- \ell_\theta(y)}^+
        \le C_2 + u(z) + C_3 \rbr{-\log(u(z))}^+.
    \]
    The cdf of $p_\theta$ is $F_\theta(y) = e^{-u(z)}$. Since $u(Z) = - \log F_\theta(Y)$ has a unit-exponential distribution, the family $\cbr{(- \ell_\theta(Y))^+ : \theta \in \Theta, Y \sim p_\theta}$ is $L^2$-stochastically dominated, as required.
\end{proof}

\begin{proof}[Proof of \cref{ex:cmle_gpd}]
As for the GEV distribution, the restriction of the shape parameter to $\xi > -1$ ensures the continuity of the map $(\theta, y) \mapsto p_\theta(y)$ on $(-1,\infty) \times (0, \infty)$, which implies the continuity requirements in \cref{ass:conditional-mle}. The same restriction also implies the bound $p_\theta(y) \le a^{-1}$ (note that if $-1 < \xi < 0$, then $-1-1/\xi > 0$), so that on the compact subset $\Theta = \bth(\Xx \times H)$ of the parameter space, the GP density is uniformly bounded above. As for the GEV distribution, this implies the $L^2$-stochastic dominance required in \cref{ass:conditional-mle}(iv).

    To show the $L^2$-stochastic dominance in \cref{ass:conditional-mle}(iii), note that for $y > 0$ such that $1 + \xi y / a > 0$, we have
    \[
        \ell_\theta(y) = \log p_\theta(y) = - \log a + (\xi + 1) \log(1 - F_\theta(y))
    \]
    where $F_\theta$ is the cdf of $p_\theta$, given by
    \[
        F_{\theta}(y) 
        = \int_0^y p_\theta(t) \, \diff t
        = \begin{cases}
            1 - \rbr{ 1 + \xi y/a }^{-1/\xi}, & \text{if $\xi \neq 0$},\\
            1 - \exp\rbr{ - y/a }, & \text{if $\xi = 0$}.
        \end{cases}
    \]
    If $Y \sim p_\theta$, the distribution of $-\log(1 - F_\theta(Y))$ is unit-exponential. The $L^2$-stochastic dominance of $\cbr{\ell_\theta(Y) : \theta \in \Theta, Y \sim p_\theta}$ then follows from the fact that $\Theta$ is a compact subset of $(0,\infty) \times (-1,\infty)$.
\end{proof}

\begin{proof}[Proof of \cref{thm:Frechet:strong}]
    We apply \cref{cor:conditional:strong} to $Y_{n,i} = (M_{n,i} \vee 1) / a_{r_n}$ and $\Yy = (0, \infty)$. The reason to take the maximum with $1$ has to do with the $L^2$-stochastic dominance condition, see below. We have 
    \[ 
        \pr\rbr{Y_{n,i} \le y} 
        = \pr\rbr{M_{n,i} \vee 1 \le a_{r_n} y} 
        = \begin{dcases}
            F_{x_{n,i}}^{r_n}(a_{r_n}y), & \text{if $a_{r_n} y \ge 1$,} \\
            0, & \text{otherwise}.
        \end{dcases}
    \]
    The sequence of Markov kernels is thus given by
    \begin{equation} 
        \label{eq:markov-kernel-frechet}
        Q_n\rbr{x, (0, y)}
        = F_x^{r_n}\rbr{a_{r_n}y} \1_{[1/a_{r_n},\infty)}(y),
        \qquad y \in (0, \infty),
    \end{equation}
    and the limit kernel by 
    \[
        Q\rbr{x, (0, y)}
        = \Phi_{\alpha_0}(y/\sigma_{\beta_0}(x)).
    \]
    The latter is continuous as a function of $x$, by continuity of the map $(x, \beta) \mapsto \sigma_\beta(x)$, while the continuous convergence of $Q_n$ to $Q$ follows from \cref{eq:frechet-doa-uniform}: for $y \in (0, \infty)$ and every sequence $x_n \to x$ in $\Xx$, we have
    \begin{align*}
        &\phantom{{}={}} \big| Q_n\rbr{x_n, (0, y)} - Q\rbr{x, (0, y)} \big|  \\
        &\le \big| F_{x_n}^{r_n}\rbr{a_{r_n}y} \1_{[1/a_{r_n},\infty)}(y) - \Phi_{\alpha_0}(y/\sigma_{\beta_0}(x_n)) \big|
        + \big| \Phi_{\alpha_0}(y/\sigma_{\beta_0}(x_n)) - \Phi_{\alpha_0}(y/\sigma_{\beta_0}(x)) \big|,
    \end{align*}
    and both terms on the right-hand side converge to zero, the former by the uniformity in $x\in\Xx$ in \cref{eq:frechet-doa-uniform} and the fact that $a_r \to \infty$ as $r\to\infty$.

    We abbreviate $r = r_n$ and write $\ell_{\tau, \alpha} = \log p_{\tau, \alpha}$. We introduce the scaling sequence $a_r$ into the objective function in \cref{eq:Fr:ll}, writing it as
    \[
        L_n(\alpha, \beta, \tau)
        = \frac{1}{n} \sum_{i=1}^n \ell_{\tau \sigma_\beta(x_{n,i}), \alpha}(M_{n,i})
        = \frac{1}{n} \sum_{i=1}^n \ell_{(\tau/a_r) \sigma_\beta(x_{n,i}), \alpha}(M_{n,i}/a_{r}) - \log a_r.
    \]
    Define $H := A \times B \times [\gamma_-, \gamma_+]$ and, for $(\alpha, \beta, \gamma) \in H$,
    \[
        \tilde{L}_n(\alpha, \beta, \gamma)
        := \frac{1}{n} \sum_{i=1}^n 
        \ell_{\gamma \sigma_\beta(x_{n,i}), \alpha}(Y_{n,i}).
    \]
    Thanks to \cref{lem:frechet-omega-0-probability-1}, the event that there exists (random) $n_0 \in \nat$ such that $Y_{n,i} = M_{n,i}/a_r$ for all $n \ge n_0$ and $i \in [n]$ has probability one. Write $\tilde{\eta}_n = (\hat{\alpha}_n, \hat{\beta}_n, \hat{\gamma}_n)$ with $\hat{\gamma}_n = \hat{\tau}_n / a_r$, and note that $\tilde{\eta}_n$ takes its values in $H$. By the assumption on $\hetan$, we have, almost surely,
    \begin{align*}
        \tilde{L}_n(\tilde{\eta}_n)
        = \frac{1}{n} \sum_{i=1}^n 
        \ell_{(\hat{\tau}_n/a_r) \sigma_{\hat{\beta}_n}(x_{n,i}), \hat{\alpha}_n}(Y_{n,i}) 
        &= \frac{1}{n} \sum_{i=1}^n 
        \ell_{(\hat{\tau}_n/a_r) \sigma_{\hat{\beta}_n}(x_{n,i}), \hat{\alpha}_n}(M_{n,i}/a_r) + \oh(1) \\
        &= L_n(\hetan) + \log a_r + \oh(1) \\
        &\ge L_n(\alpha_0, \beta_0, a_r) + \log a_r + \oh(1) \\
        &= \tilde{L}_n(\alpha_0, \beta_0, 1) + \oh(1),
        \qquad n \to \infty.
    \end{align*}
    The conclusion now follows if we can apply \cref{cor:conditional:strong} with contrast function 
    \begin{align*}
        m_\eta(x,y) 
        &= \ell_{\gamma \sigma_\beta(x), \alpha}(y) \\
        &= \log \alpha - \alpha \log \gamma - \alpha \log \sigma_\beta(x) - (\alpha+1) \log y - \rbr{\frac{y}{\gamma \sigma_\beta(x)}}^{-\alpha}.
    \end{align*}
    with $\eta = (\alpha, \beta, \gamma) \in H$ and true parameter $\eta_0 = (\alpha_0, \beta_0, 1)$. 

    \cref{ass:dgpZ-conditional} has already been verified at the beginning of the proof.

    The map $(\eta, x, y) \mapsto m_\eta(x,y)$ is continuous and bounded above, so $m_\eta$ is $P_{X,Y}$-quasi-integrable, which is \cref{ass:critZ}.

    In \cref{ass:trueZ}, it only remains to show that $\eta_0 = (\alpha_0, \beta_0, 1)$ is the unique maximizer in $H$ of the limit function
    \[
        \tilde{L}(\eta) 
        = P_{X,Y} m_{\eta} 
        = \int_{\Xx} \int_{0}^{\infty} 
            m_{\eta}(x, y) \, p_{\sigma_{\beta_0}(x), \alpha_0}(y) \, \diff y \, 
        P_X(\diff x).
    \]
    But this follows from the identifiablity of the Fréchet parameter $(\sigma, \alpha)$ together with the identifiability assumption in \cref{eq:Fr:ident}: for all $x \in \Xx$, we have
    \[
        \int_{0}^{\infty} 
            m_{\eta}(x, y) \, p_{\sigma_{\beta_0}(x), \alpha_0}(y) \, \diff y
        \le
        \int_{0}^{\infty} 
            m_{\eta_0}(x, y) \, p_{\sigma_{\beta_0}(x), \alpha_0}(y) \, \diff y
    \]
    with strict inequality as soon as $\eta=(\alpha, \beta, \gamma)$ satisfies $(\gamma \sigma_\beta(x), \alpha) \ne (\sigma_{\beta_0}(x), \alpha_0)$ \citep[Lemma 5.35]{van1998asymptotic}, which happens on a set of positive $P_X$-probability. Further, it is not hard to see that $\tilde{L}(\eta_0) > -\infty$, since if $Y \sim \Phi_\alpha$, then $Y^{-\alpha}$ has a unit-exponential distribution.

    Finally, consider the $L^2$-stochastic dominance in \cref{ass:stochdomZ}. Part (ii) follows from the fact $(\eta,x,y) \mapsto m_\eta(x,y)$ is upper bounded, and part (i) follows from \cref{lem:FrechetL2}, observing that $Z_{n,i}=(x_{n,i}, Y_{n,i}) \sim Q_n(x, \diff y) \delta_{x_{n,i}}(\diff x)$ with $Q_n$ from \eqref{eq:markov-kernel-frechet}. 
\end{proof}

\begin{lemma}
\label{lem:frechet-omega-0-probability-1}
    Assume \cref{eq:F0DoA,eq:FxF0} with $\inf_{x\in\Xx} c(x) > 0$. If $\log n = o(r_n)$, then for $M_{n,i}$ in \cref{eq:Mnixi} we have $\min_{i\in[n]} M_{n,i} \to \infty$ almost surely as $n \to \infty$.
\end{lemma}

\begin{proof}[Proof of \cref{lem:frechet-omega-0-probability-1}]
    Choose $y < \infty$ arbitrarily large. It is sufficient to show that 
    the event $\{ \min_{i\in[n]} M_{n,i} \le y \text{ infinitely often} \}$  has probability zero.
    Let $\eps > 0$ be such that $c(x) \ge 2\eps$ for all $x \in \Xx$. By enlarging $y$ if necessary, we can ensure that
    \[
        \sup_{x \in \Xx} \abr{\frac{\log F_x(y)}{\log F_0(y)} - c(x)} \le \eps
    \]
    and thus
    \[
        \forall x \in \Xx: \qquad 
        F_x(y) \le F_0(y)^{c(x)-\eps} \le F_0(y)^\eps < 1,
    \]
    as the upper endpoint of $F_0$ is infinity by \cref{eq:F0DoA}. Write $\bar{p} = \sup_{x\in\Xx} F_x(y) < 1$ and let $n$ be sufficiently large so that $\log n \le \abr{\log \bar{p}} r_n/3$. Then
    \begin{align*}
        \pr \Big(\min_{i\in[n]} M_{n,i} \le y\Big)
        &\le \sum_{i=1}^n \pr\rbr{M_{n,i} \le y} \\
        &\le n \sup_{x \in \Xx} F_x^{r_n}(y) 
        = n\exp\rbr{ - r_n \abr{\log \bar{p}}} 
        \le n\exp\rbr{-3\log n} =n^{-2}.
    \end{align*}
    The upper bound is summable in $n$, so that the Borel--Cantelli permits to conclude.
\end{proof}

\begin{lemma} 
\label{lem:FrechetL2}
Under the conditions of \cref{thm:Frechet:strong}, there exists $n_0 \in \nat$ such that the family
\[
    \bigl\{ m_{\eta_0}(x,Y): Y \sim Q_n(x, \point), x \in \Xx, n \ge n_0\bigr\},
\]
with $Q_n$ as in \cref{eq:markov-kernel-frechet}, is $L^2$-stochastically dominated.
\end{lemma}

\begin{proof}[Proof of \cref{lem:FrechetL2}] 
Note that
\[
m_{\eta_0}(x,y) = \log \alpha_0 - \log \sigma_{\beta_0}(x) - (\alpha_0+1) \log \Big( \frac{y}{\sigma_\beta(x)}\Big) - \Big( \frac{y}{\sigma_\beta(x)}\Big)^{-\alpha_0}.
\]
Since $L^2$-stochastic dominance is preserved under addition, it is sufficient find $n_0=n_0(j)$ such that
\[
D_j := \big\{ f_j(x,Y): Y \sim Q_n(x, \point), x \in \Xx, n \ge n_0\big\}, \qquad j \in \{1,2,3\},
\]
is $L^2$-stochastically dominated,
where
\[
f_1(x,y) = \log_+(y/\sigma_{\beta_0}(x)), 
\quad 
f_2(x,y) = \log_-(y/\sigma_{\beta_0}(x)),
\quad 
f_3(x,y) = (y/\sigma_{\beta_0}(x))^{-\alpha_0},
\]
and where $\log_+(x) = \max\{\log(x), 0\}$ and $\log_-(x) = \max\{-\log(x), 0\}$.
For that purpose, writing $r=r_n$ and recalling that the support of $Q_n$ is $[1/a_{r},\infty)$, we need to bound
\[
t \mapsto \sup_{n \ge n_0} \sup_{x \in \Xx} 
S_{nj}(x,t) \quad \text{with}\quad  S_{nj}(x,t) := Q_n \Big( x, \big\{ y \ge 1/a_{r} : f_j(x,y) > t \big\} \Big)  
\]
by a survival function associated with a random variable that is square-integrable.

We start by the hardest case, $j=3$. In view of the fact that $\sigma_+ = \sup_{x \in \Xx} \sigma_{\beta_0}(x) \in (0,\infty)$, we have
\begin{align*}
S_{n3}(x,t)
&=
Q_n \Big( x, \big\{ 1/a_r \le y < \sigma_{\beta_0}(x) t^{-1/\alpha_0} \big\} \Big)
\\ &\le
Q_n \Big( x, \big\{ 1/a_r \le y < \sigma_{+} t^{-1/\alpha_0} \big\} \Big) =: S_{n3}^+(x,t).
\end{align*}
For some constant $c>1$ that will be specified below, write $b_r^- = (\tilde a_r/c)^{\alpha_0}$ and $b_r^+ = \tilde a_r^{\alpha_0}$,  where $\tilde a_r = a_r\sigma_+$. We will distinguish three cases according to whether $(n,x,t)$ satisfies $t \in [1, b_r^-]$, $t \in (b_r^-, b_r^+]$ or $t>b_r^+$.

In the last case, $t>b_r^+$, we have $1 / a_r > \sigma_{+}t^{-1/\alpha_0}$, which immediately implies $S_{n3}^+(x,t)=0$.
For all $1 \le t \le b_r^+$ we have $1 / a_r \le \sigma_{+}t^{-1/\alpha_0}$, so that we may write
\begin{align} \label{eq:frechet-sn3-expansion-new}
S_{n3}^+(x,t)
&= \nonumber
F_x^{r}(\tilde a_r t^{-1/\alpha_0}) 
\\&= 
\exp\Big( 
    - \big\{ - r \log F_0(\tilde a_r) \big\} \cdot \frac{\log F_0(\tilde a_r t^{-1/\alpha_0})}{\log F_0(\tilde a_r)} \cdot
    \frac{\log F_x(\tilde a_rt^{-1/\alpha_0})}{\log F_0(\tilde a_r t^{-1/\alpha_0})} \Big).
\end{align}

We next consider the case $t \in [1, b_r^-]$. First, 
\cref{eq:F0DoA} implies $-r \log F_0(\tilde a_r) = (\sigma_+)^{-\alpha_0} + \oh(1)$, whence there exists $n_1\in \nat$ such that $-r \log F_0(\tilde a_r) \ge (\sigma_+)^{-\alpha_0}/2$ for all $n \ge n_1$. Next, in view of \cref{eq:FxF0} there exists $c$ (wlog larger than $1$) such that  
\[
\sup_{x \in \Xx}\Big| \frac{\log F_x(z)}{\log F_0(z)} - (\sigma_{\beta_0}(x))^{\alpha_0} \Big| \le (\sigma_-)^{\alpha_0}/2 \qquad \forall z \ge c,
\] 
where $\sigma_- = \inf_{x \in \Xx} \sigma_{\beta_0}(x) \in (0,\infty)$.
Since $t \le b_r^-$ is equivalent to $\tilde a_r t^{-1/\alpha_0} \ge c$, we obtain that
\[
\frac{\log F_x(\tilde a_rt^{-1/\alpha_0})}{\log F_0(\tilde a_r t^{-1/\alpha_0})}
\ge 
(\sigma_{\beta_0}(x))^{\alpha_0} - \Big|  \frac{\log F_x(\tilde a_rt^{-1/\alpha_0})}{\log F_0(\tilde a_r t^{-1/\alpha_0})} - (\sigma_{\beta_0}(x))^{\alpha_0} \Big| \ge (\sigma_-)^{\alpha_0}/2
\]
for all $x \in \Xx$ and all $(n,t)$ such that $t \in [1, b_r^-]$. Finally, fix $\delta \in (0, 1 \wedge \alpha_0)$. After possibly increasing the constant $c$, the Potter bounds on the regularly varying function $-\log F_0$ \citep[Theorem~1.5.6]{BGT87} imply that, for all $u,v\ge c$,
\begin{align} \label{eq:frechet-potter-uniform-new}
\frac{-\log F_0(u)}{- \log F_0(v) } \le (1+\delta) \Big( \frac{u}{v}\Big)^{-\alpha_0} \max \Big[ \Big(\frac{u}{v} \Big)^{\delta}, \Big(\frac{u}{v} \Big)^{-\delta} \Big].
\end{align}
Since $a_r \to \infty$, we may choose $n_2$ such that $\tilde a_r \ge c$ for all $n \ge n_2$. We may apply the previous inequality with $u=\tilde a_r$ and $v= \tilde a_r y$ for $y \in [c/ \tilde a_r,1]$ to deduce that
\[
\frac{-\log F_0(\tilde a_r)}{- \log F_0(\tilde a_r y) } \le (1+\delta) y^{\alpha_0-\delta}
\qquad \forall n \ge n_2, y \in [c/\tilde a_r,1].
\]
This in turn implies
\[
\frac{-\log F_0(\tilde a_r y)}{- \log F_0(\tilde a_r) } \ge \frac1{1+\delta} y^{-\alpha_0+\delta}
\qquad \forall  n \ge n_2, y \in [c/\tilde a_r,1],
\]
and after setting $y=t^{-1/\alpha_0}$ as required in \eqref{eq:frechet-sn3-expansion-new} and recalling $b_r^- = (\tilde a_r /c)^{\alpha_0}$, we obtain that
\[
\frac{-\log F_0(\tilde a_r t^{-1/\alpha_0})}{- \log F_0(\tilde a_r) } \ge \frac1{1+\delta} t^{1- \delta/\alpha_0}
\qquad \forall n \ge n_2, t \in [1,b_r^-].
\]
Overall, defining $n_0 = \max(n_1, n_2)$, we have shown that
\[
S_{n3}^+(x,t) \le \exp\big( - k_\delta t^{1- \delta/\alpha_0} \big) \qquad  \forall x \in \Xx, n \ge n_0, t \in  [1,b_r^-],
\]
where $k_\delta=(\sigma_-/\sigma_+)^{\alpha_0}/\{4(1+\delta)\}$.

It remains to consider the case $t \in (b_r^-, b_r^+]$. Since $t \mapsto S_{n3}^+(x,t)$ is decreasing, the previous display implies that
\begin{align*}
S_{n3}^+(x,t) 
\le 
S_{n3}^+(x, b_r^-) 
&\le 
\exp\Big( - k_\delta (b_r^-)^{1- \delta/\alpha_0} \Big)
\\&=
\exp\Big( - k_\delta\Big(\frac{b_r^-}t\Big)^{1- \delta/\alpha_0} t^{1- \delta/\alpha_0}  \Big)
\\&\le 
\exp\Big( -k_\delta \Big(\frac{b_r^-}{b_r^+}\Big)^{1- \delta/\alpha_0} t^{1- \delta/\alpha_0}  \Big)
\\&=
\exp\Big( -k_\delta\Big(\frac{1}{c}\Big)^{\alpha_0- \delta} t^{1- \delta/\alpha_0}  \Big) 
\qquad \forall x \in \Xx,  n \ge n_0, t \in  (b_r^-, b_r^+].
\end{align*}
This bound is also an upper bound for the bounds that have been derived for the cases $t \in [1, b_r^-]$ and  $t>b_r^+$, whence we have overall shown that
\begin{align*}
S_{n3}^+(x,t)
\le
\exp\Big( - k_\delta \Big(\frac{1}{c}\Big)^{\alpha_0- \delta} t^{1- \delta/\alpha_0}  \Big) 
\qquad \forall x \in \Xx,  n \ge n_0, t \ge 1.
\end{align*}
The right-hand side defines a survival function associated with a random variable that is easily square integrable (recall that $\delta<\alpha_0$).

Next, consider the case $j=2$. Since $\log_-(y) \le y^{-1}$ for all $y>0$, we have $S_{n2}(x,t) \le S_{n3}(x,t^{\alpha_0})$. We may hence reuse the bound obtained for $S_{n3}$.

Finally, consider the case $j=1$. Observing that $\log_+(y)>t$ is equivalent to $y>e^t$ (for $t\ge 0$), we have
\begin{align*}
S_{n1}(x,t) &=  
Q_n  \Big( x, \big\{ y \ge 1/a_r: {y} > \sigma_{\beta_0}(x)e^t \big\} \Big)
\le
Q_n  \Big( x, \big\{ y : {y} > \sigma_{\beta_0}(x)e^t \big\} \Big).
\end{align*}
By definition of $Q_n$, we may rewrite the right-hand side of the previous display as
\[
\pr\rbr{ \frac{\max(\xi_1^x, \dots, \xi_{r}^x) \vee 1}{a_r} > \sigma_{\beta_0}(x)e^t},
\]
where $\xi_1^x, \xi_2^x, \dots$ are iid random variables with cdf $F_x$.
Choose $n_3$ sufficiently large such that $1/a_r \le \sigma_-$ for all $n \ge n_3$. Then $1/a_r \le  \sigma_{\beta_0}(x)e^t$ for all $t \ge 0, x \in \Xx$ and $n\ge n_3$, whence the previous probability is bounded by
\begin{align*}
\pr\rbr{ \frac{\max(\xi_1^x, \dots, \xi_{r}^x)}{a_r} > \sigma_{\beta_0}(x)e^t }
\le 
r \pr\big( \xi_1^x > a_r \sigma_{\beta_0}(x)e^t \big)
&\le 
r \pr\big( \xi_1^x > a_r \sigma_{-}e^t \big)
\\&=
r \{ 1- F_x(a_r \sigma_{-}e^t )\}
\\&\le 
- r  \log F_x(a_r \sigma_{-}e^t ),
\end{align*}
where the last inequality follows from $1-z \le - \log z$ for all $z \in [0,1]$. Next, rewrite the right-hand side of the previous display as
\[
- r  \log F_0(a_r \sigma_{-}) \cdot \frac{\log F_0(a_r \sigma_{-}e^t )}{\log F_0(a_r \sigma_{-})} \cdot \frac{\log F_x(a_r \sigma_{-}e^t )}{\log F_0(a_r \sigma_{-}e^t )}.
\]
\cref{eq:F0DoA} implies $-r \log F_0(a_r \sigma_{-} ) = (\sigma_-)^{-\alpha_0} + \oh(1)$, whence there exists $n_3\in \nat$ such that $0 \le -r \log F_0(a_r \sigma_{-}) \le 2(\sigma_-)^{-\alpha_0}$ for all $n \ge n_3$. Next, \cref{eq:frechet-potter-uniform-new} applied with $u=a_r\sigma_-e^t$ and $v = a_r\sigma_-$ yields
\[
\frac{\log F_0(a_r \sigma_{-}e^t )}{\log F_0(a_r \sigma_{-})}
\le 
(1+\delta) (e^t)^{-\alpha_0+\delta}=(1+\delta) e^{-t(\alpha_0-\delta)}
\]
for all $n\ge n_4$ and $t\ge 0$; here, $n_4$ is chosen such that $a_r\sigma_- \ge c$ for all $n \ge n_4$.
Finally, by \cref{eq:FxF0}, there exists $n_5\in\nat$ such that for all $n \ge n_5$, all $x \in \Xx$  and all $t \ge 0$,
\[
0 \le \frac{\log F_x(a_r \sigma_{-}e^t )}{\log F_0(a_r \sigma_{-}e^t )} \le 2(\sigma_+)^{\alpha_0}.
\]
Overall, we have shown that
\[
S_{n1}(x,t) \le  4(1+\delta) (\sigma_+/\sigma_-)^{\alpha_0} e^{-t(\alpha_0-\delta)} \qquad \forall x \in \Xx, n \ge n_3 \vee n_4 \vee n_5, t \ge 0,
\]
and since $\delta<\alpha_0$, the proof is finished.
\end{proof}

\subsection{Proofs for Section~\ref{sec:argmax}}
\label{sec:argmax-proofs}

\begin{proof}[Proof of \cref{thm:lachout-modified}] 
Fix $\delta_0>0$.  By continuity of measure, it is sufficient to show that, with probability one, `$d_H(\hat \eta_n, \eta_0) < \delta_0$ for all sufficiently large $n$'. Equivalently, defining $\Delta := \{ \eta \in H \colon d_H(\eta, \eta_0) \geq \delta_0 \}$, we need to show that `$\hat \eta_n \notin \Delta$ for all sufficiently large $n$' with probability one.

For that purpose, note that, for any $\eta \ne \eta_0$, we have
\[ 
    \lim_{\delta \downarrow 0}  P m_{B(\eta, \delta)}    =  M(\eta) < M(\eta_0) < \infty
\]
by \cref{lem:ext_ballConv-abstract} and condition~\ref{L-identifiable}, where \cref{lem:ext_ballConv-abstract} is applicable by \ref{L-semicontinuous} and the first part of \ref{L-bounded-integrals-convergence-suprema}. Hence, after possibly decreasing $\delta = \delta(\eta) > 0$ from \ref{L-bounded-integrals-convergence-suprema}, we can assume that
\begin{equation}
\label{eq:proofCons1a-abstract}
    P m_{\clB(\eta,\delta)} < M(\eta_0).
\end{equation}
We may cover the set $\Delta$ by the collection of open balls $\{B(\eta, \delta(\eta)): \eta \in \Delta\}$. Since $\Delta$ is compact by compactness of $H$, we can choose a finite subcover, say $\Delta \subset \bigcup_{i=1}^K B_i$ where $B_i = B(\eta_i, \delta(\eta_i))$.

Now, for each $i=1, \dots, K$,
\begin{align*}
\limsup_{n\to\infty} \Pemp m_{B_i} 
\le
\limsup_{n\to\infty} \Pemp m_{\closure(B_i)} 
&\le 
P m_{\closure(B_i)}
< M(\eta_0)
\end{align*}
with probability one, where we used condition~\ref{L-bounded-integrals-convergence-suprema} and \cref{eq:proofCons1a-abstract} at the second and third inequality, respectively. As a consequence, 
\[
\limsup_{n\to\infty} \sup_{\eta \in \Delta} \Pemp m_\eta
\le 
\limsup_{n\to\infty} \max_{i=1, \dots, K} \Pemp m_{B_i} 
< M(\eta_0)
\]
with probability one. 

Also, by condition~\ref{L-convergence-eta0} and since $M_n(\hetan) \geq M_n(\eta_0) - \oh_{a.s.}(1)$, we further have
\[
M(\eta_0) \le \liminf_{n \to \infty} M_n(\eta_0) \le \liminf_{n\to\infty} M_n(\hat \eta_n)   = \liminf_{n \to \infty} \Pemp m_{\hat \eta_n}
\]
with probability one. Let $\Omega_0$ denote the intersection of the two previous probability-one events, and define $N=\{\hetan \in \Delta$ infinitely often$\}$. The event $N \cap \Omega_0$ is then empty: indeed, on $N \cap \Omega_0$ we would have
\[
    \liminf_{n\to\infty} \Pemp m_{\hat{\eta}_n}
    \le \limsup_{n\to\infty} \sup_{\eta\in\Delta} \Pemp m_\eta
    < M(\eta_0)
    \le \liminf_{n\to\infty} \Pemp m_{\hat{\eta}_n};
\]
a contradiction. Hence, $\pr(N) = \pr(N \cap \Omega_0^c) \le \pr(\Omega_0^c)=0$, which was to be shown.
\end{proof}

\begin{proof}[Proof of \cref{thm:lachout-modified:weak}]
The first two paragraphs of the proof are the same as those of the proof of Theorem~\ref{thm:lachout-modified}.
On the one hand, we have
\begin{equation} 
    \label{eq:lachout-modified:weak:proof1}
    M_n(\hetan) 
    \ge M_n(\eta_0) - \oh_{\pr}(1) 
    \ge M(\eta_0) - \oh_{\pr}(1),
    \qquad n \to \infty,
\end{equation}
by the assumption on the estimator and by Assumption~\ref{L-convergence-eta0-P}. On the other hand, with $\delta(\eta_i)$ and $B_i$ for $i = 1,\ldots,K$ defined as in the proof of \cref{thm:lachout-modified}, we have
\begin{align*}
    \Pemp m_{B_i} 
    &\le \Pemp m_{\closure(B_i)} 
    \le P m_{\closure(B_i)} + \oh_{\pr}(1) 
    = P m_{\clB(\eta_i,\delta(\eta_i))} + \oh_{\pr}(1),
    \qquad n \to \infty,
\end{align*}
by Assumption~\ref{L-convergence-suprema-P} and thus 
\begin{align}
    \nonumber
    \sup_{\eta \in \Delta} M_n(\eta)
    &\le \max_{i = 1,\ldots,K} \Pemp m_{B_i} \\
    \label{eq:lachout-modified:weak:proof2}
    &\le \max_{i = 1,\ldots,K} P m_{\clB(\eta_i,\delta(\eta_i))} + \oh_{\pr}(1) 
    \le M(\eta_0) - \eps + \oh_{\pr}(1) 
\end{align}
for some $\eps > 0$, since $P m_{B(\eta_i,\delta(\eta_i))} < M(\eta_0)$ for all $i=1,\ldots,K$. Together, \cref{eq:lachout-modified:weak:proof1,eq:lachout-modified:weak:proof2} imply that the event $\{\hetan \in \Delta\}$ is included in the event $\{M(\eta_0) - \oh_{\pr}(1) \le M(\eta_0)-\eps+\oh_{\pr}(1)\}$, and the probability of that event goes to zero as $n \to \infty$. Hence,
$\pr(\hetan \in \Delta) \to 0$, as required.
\end{proof}

\begin{lemma}[Convergence of suprema over shrinking balls]
\label{lem:ext_ballConv-abstract}
Suppose that Assumption \ref{L-semicontinuous} is met and let $\eta \in H$ be arbitrary. If $P m_{\clB(\eta,\delta)} < \infty$ for some $\delta>0$, then
\[
	\lim_{\dlt \downarrow 0}
	P m_{B(\eta,\dlt)}
	=
	P m_\eta.
\]
\end{lemma}

\begin{proof} 
Fix $\eta \in H$, and let $\delta_n \downarrow 0$ as $n \to \infty$. For $z\in \Ee$, define $f_n(z) = m_{B(\eta,\delta_n)}(z)$, and note that there exists $n_0$ such that $P f_n \le P f_{n_0} < \infty$ for all $n \ge n_0$ by monotonicity and the assumption that $P m_{\clB(\eta,\delta)} < \infty$ for some $\delta>0$.
If $P f_{n_0} = -\infty$, there is nothing to show, so assume $P f_{n_0} > -\infty$. By the monotone convergence theorem
applied to the non-decreasing sequence $(f_{n_0}-f_n) \1_{\{|f_{n_0}|<+\infty\}}$,
\[
\lim_{n\to \infty} P f_n = P \Big( \lim_{n \to \infty} f_n \Big) =  \int_\Ee \lim_{n \to \infty} \sup_{\eta' \in B(\eta, \delta_n)} m_{\eta'}(z) \, P(\diff z).
\]
By Assumption~\ref{L-semicontinuous}, the integration domain on the right-hand side can be replaced by $\Ee \setminus V_\eta$, and on that domain, we have 
\[
\lim_{n \to \infty} \sup_{\eta' \in B(\eta, \delta_n)} m_{\eta'}(z)  = m_\eta(z) 
\]
for all $z \in \Ee \setminus V_\eta$ by upper semicontinuity. This implies the assertion.
\end{proof}

\bibliographystyle{chicago}
\bibliography{biblio}
\end{document}